\RequirePackage{fix-cm}
\documentclass[11pt]{article}

\usepackage[margin=1in]{geometry}
\usepackage[T1]{fontenc}
\usepackage{lmodern}
\usepackage{amsmath,amssymb,amsfonts,amsthm,mathtools}
\usepackage[numbers,sort&compress]{natbib}
\usepackage{xcolor}
\usepackage{todonotes}
\usepackage{graphicx}
\usepackage{comment}
\usepackage{subcaption}
\usepackage{tcolorbox}
\usepackage{multirow}
\usepackage{booktabs}
\usepackage{array}
\usepackage{makecell}
\usepackage{pifont}
\usepackage{adjustbox}
\usepackage{cases}
\usepackage{algorithm}
\usepackage{algpseudocode}
\usepackage{hyperref}
\usepackage[capitalise,nameinlink]{cleveref}

\newcommand{\cahiernumber}{56}  
\newcommand{\cahieryear}{\the\year}   

\usepackage{fancyhdr}
\fancypagestyle{cahier}{%
  \fancyhf{}%
  \fancyhead[R]{\normalsize Cahier du GERAD G-\cahieryear-\cahiernumber}%
  \fancyfoot[C]{\thepage}%
}

\newcommand{\cmark}{\ding{51}} 
\newcommand{\xmark}{\ding{55}} 

\theoremstyle{plain}
\newtheorem{theorem}{Theorem}[section]
\newtheorem{lemma}[theorem]{Lemma}

\newtheorem{proposition}[theorem]{Proposition}
\theoremstyle{definition}
\newtheorem{definition}[theorem]{Definition}

\newtheorem{problemassumption}{Problem Assumption}[section]
\newtheorem{modelassumption}{Model Assumption}[section]

\theoremstyle{remark}

\theoremstyle{definition}
\newtheorem{example}[theorem]{Example}
\crefname{example}{Example}{Examples}
\Crefname{example}{Example}{Examples}
\crefname{remark}{Remark}{Remarks}
\Crefname{remark}{Remark}{Remarks}
\crefname{appendix}{Appendix}{Appendices}
\Crefname{appendix}{Appendix}{Appendices}
\crefname{problemassumption}{Problem Assumption}{Problem Assumptions}
\Crefname{problemassumption}{Problem Assumption}{Problem Assumptions}
\crefname{modelassumption}{Model Assumption}{Model Assumptions}
\Crefname{modelassumption}{Model Assumption}{Model Assumptions}
\crefname{stepassumption}{Step Assumption}{Step Assumptions}
\Crefname{stepassumption}{Step Assumption}{Step Assumptions}
\providecommand{\R}{\mathbb{R}}
\DeclareMathOperator*{\minimize}{minimize}

\newcommand{\mailto}[1]{\href{mailto:#1}{\nolinkurl{#1}}}
\newenvironment{keywords}{\par\medskip\noindent\textbf{Keywords:}}{\par}
\newenvironment{AMS}{\par\noindent\textbf{AMS subject classifications:}}{\par}
\providecommand{\headers}[2]{}

\usepackage[
  createShortEnv,
  conf={one big link},
  commandRef=Cref]{proof-at-the-end}

\NewDocumentEnvironment{proofEE}{O{}+b}{%
  \begin{proofE}[#1]
    #2
    \space
  \end{proofE}
}{}

\usepackage{tikz}
\usepackage{pgfplots}
\usepackage{tikzscale}
\usetikzlibrary{external}
\def\figurespath{figures}
\graphicspath{{\figurespath/}}
\tikzset{external/system call={lualatex \tikzexternalcheckshellescape
      -halt-on-error
      -interaction=batchmode
      -jobname "\image" "\texsource"}}

\makeatletter
\renewcommand{\todo}[2][]{\tikzexternaldisable\@todo[#1]{#2}\tikzexternalenable}
\newcommand*{\defeq}{\mathrel{\vcenter{\baselineskip0.5ex \lineskiplimit0pt
                     \hbox{\scriptsize.}\hbox{\scriptsize.}}}%
                     =}
\makeatother

\newcommand{\kdc}{\kappa_{\textup{mdc}}}

\newcommand{\tl}[1]{L_#1}

\newcommand{\galpha}{g}
\newcommand{\Psialpha}{\Psi}
\newcommand{\psialpha}{\psi}
\newcommand{\halpha}{h}
\providecommand{\N}{\mathbb{N}}

\def\papertitle{Complexity analysis of trust-region methods under \((\alpha,L_0,L_1)\)-smoothness}
\def\authorone{Youssef Diouane}
\def\authortwo{Mohamed L. Habiboullah}
\def\authorthree{Awa Khouna}
\def\authorfour{Dominique Orban}

\hypersetup{
  colorlinks=true, linkcolor=blue, citecolor=blue, urlcolor=blue,
  pdftitle={Complexity analysis of trust-region methods under (alpha,L0,L1)-smoothness},
  pdfauthor={\authorone, \authortwo, \authorthree, and \authorfour},
  pdfsubject={Trust-region methods under the (alpha,L0,L1)-smoothness condition},
  pdfkeywords={Complexity analysis, Quasi-Newton methods, Trust-region methods, Unconstrained optimization, Nonconvex optimization, Convex optimization},
}

\title{\papertitle}

\author{
  \authorone\thanks{%
    GERAD and Department of Mathematics and Industrial Engineering, Polytechnique Montr\'eal.
    E-mail: \mailto{youssef.diouane@polymtl.ca}.
    Research partially supported by an NSERC Discovery Grant.}
  \and
  \authortwo\thanks{%
    GERAD and Department of Mathematics and Industrial Engineering, Polytechnique Montr\'eal.
    E-mail: \mailto{mohamed-laghdaf-2.habiboullah@polymtl.ca}.
    Research partially supported by an FRQNT scholarship.}
  \and
  \authorthree\thanks{%
    CIRRELT and SCALE-AI Chair in Data-Driven Supply Chains, Department of Mathematics and Industrial Engineering, Polytechnique Montr\'eal.
    E-mail: \mailto{awa.khouna@polymtl.ca}.}
  \and
  \authorfour\thanks{%
    GERAD and Department of Mathematics and Industrial Engineering, Polytechnique Montr\'eal.
    E-mail: \mailto{dominique.orban@gerad.ca}.
    Research partially supported by an NSERC Discovery Grant.}
}
\date{\today}

\begin{document}

\maketitle
\thispagestyle{cahier}

\begin{abstract}
  Generalized smoothness assumptions have attracted growing attention in recent years, motivated in part by machine learning problems in which the gradient of the objective function may not be Lipschitz continuous.
  Among the most prominent of these is the $(\alpha, L_0, L_1)$-smoothness assumption.
  Existing methods that attain the best known complexity bounds require knowledge of, or upper bounds on, $\alpha$, $L_0$ and $L_1$, while the few parameter-agnostic methods available do not recover those bounds.
  In this paper, we show that trust-region methods attain the best known bounds for $(\alpha, L_0, L_1)$-smooth objective functions without prior knowledge of these parameters.
  Establishing these results requires new analytical tools, beyond those used in classical trust-region complexity analyses.
  We further show that our working model assumption allows the use of general model Hessian approximations and, in particular, accommodates limited-memory quasi-Newton updates.
  Finally, we show that our complexity bound is sharp in the nonconvex setting.
\end{abstract}



\section{Introduction}%
\label{sec:intro}

We consider the unconstrained optimization problem
\begin{equation}%
  \label{eq:nlp}
  \minimize_{x \in \R^n} \ f(x),
\end{equation}
where \(f:\mathbb R^n\to\mathbb R\) is \(\mathcal C^1\) and bounded below by \(f_{\rm low}\).
We target to solve~\eqref{eq:nlp} by means of a trust-region method~\citep{conn-gould-toint-2000}.
Given a tolerance \(\epsilon>0\), our goal is to compute an \(\epsilon\)-approximate first-order stationary point \(x\) satisfying \(\|\nabla f(x)\|\le\epsilon\).
The iteration complexity of an optimization method is the worst-case number of iterations required to reach such a point, while its evaluation complexity counts the associated objective and derivative evaluations.
In this paper, we study the evaluation complexity of trust-region methods under generalized smoothness assumptions introduced by \citep{zhang2019gradient,chen-zhou-liang-lu-2023}, which we refer to as \((\alpha,L_0,L_1)\)-smoothness.
Specifically, assuming that the Hessian of \(f\) exists, we suppose that there exist constants \(L_0,L_1\ge0\) and \(\alpha\in[0,1]\) such that
\begin{equation}%
  \label{eq:L0L1-smoothness-hessian}
  \|\nabla^2 f(x)\|
  \le
  L_0+L_1\|\nabla f(x)\|^\alpha,
  \qquad
  \forall x\in\mathbb{R}^n.
\end{equation}
The endpoint \(\alpha=1\) is commonly referred to as \((L_0,L_1)\)-smoothness and has been studied extensively \citep{zhang2019gradient,zhang2020improved,koloskova2023revisiting,gorbunov2025methods,vankov2024optimizing,hubler2024parameteragnostic}.

Classical worst-case complexity analyses for smooth optimization usually rely on global Lipschitz continuity of the gradient, a convenient assumption that yields the standard sharp \(\mathcal{O}(L\epsilon^{-2})\) complexity bound for nonconvex \(f\), where \(L\) is the Lipschitz constant of the gradient \citep{cartis-gould-toint-2022,grapiglia-2014}.
However, \(L\) may be large in practice, which can lead to a pessimistic bound.
Condition~\eqref{eq:L0L1-smoothness-hessian} generalizes the classical setting, since \(\nabla f\) is \(L\)-Lipschitz if \(f\) is, e.g., \((0,0,L)\)-smooth.
In the training of long short-term memory networks, \(L_0\) is observed to be much smaller than \(L\) \citep{zhang2019gradient,zhang2020improved}, and other machine learning problems exhibit the same behavior \citep{bodard2026escaping}.
These observations have motivated the study of complexity bounds under \((\alpha,L_0,L_1)\)-smoothness.

\subsection*{Contributions}
This paper makes the following contributions:

\begin{itemize}
  \item We develop a complexity analysis of trust-region methods under \((\alpha ,L_0, L_1)\)-smoothness for nonconvex, convex, and strongly convex objectives.
    Unlike the classical analysis, our approach requires new mathematical tools to recover the best-known complexity bounds.
  \item We recover the best known complexity bounds in the nonconvex, convex, and strongly convex settings, matching the state of the art without requiring prior knowledge of \((\alpha ,L_0, L_1)\), except for those of accelerated methods in the convex case.
  \item We introduce a general class of \((\alpha,L_0,L_1)\)-smooth functions, which we call polynomial inverse problems and which includes phase retrieval as a special case.
    To the best of our knowledge, the literature contains a single explicit example for \(\alpha\in(0,1)\), namely phase retrieval itself.
  \item We show how limited-memory quasi-Newton (L-BFGS, L-SR1 and L-PSB) fit within our framework.
  \item We show that the \(\epsilon^{\alpha-2}\) complexity bound is sharp for \((\alpha,0,L_1)\)-smooth nonconvex objectives.
\end{itemize}

\subsection*{Related research}

The \((1,L_0,L_1)\)-smoothness condition was introduced by \citet{zhang2019gradient} to explain the empirical effectiveness of gradient clipping when the local smoothness varies significantly.
They showed that clipping attains the best known rate, whereas fixed-step gradient descent requires the additional assumption that the gradient be uniformly bounded, and pays for it with a constant that degrades with that bound (\Cref{tab:L0L1-complexity}).
\citet{zhang2020improved} extended the condition to merely continuously differentiable functions, and \citet{koloskova2023revisiting} allowed an arbitrary clipping radius while adding convex and strongly convex guarantees, at the cost of an additional Lipschitz assumption on the gradient.
\citet{gorbunov2025methods} sharpened the strongly convex bound and combined clipping with acceleration in the convex case, at the price of a factor exponential in \(L_1\), which \citet{vankov2024optimizing} subsequently removed while identifying stepsizes that attain the best known rates for clipped, normalized and accelerated gradient methods (\Cref{tab:L0L1-complexity}).
In every one of these methods, the stepsize or the clipping radius is prescribed from \(L_0\) and \(L_1\).

The general condition~\eqref{eq:L0L1-smoothness-hessian} for \(\alpha\in(0,1)\) was introduced by \citet{chen-zhou-liang-lu-2023}, whose normalized gradient descent is optimal at \(\alpha=1\), but whose bound for \(\alpha<1\) carries a factor that tends to infinity as \(\alpha\uparrow1\) (\Cref{tab:L0L1-complexity-alpha}).
\citet{li2023convex} works under a condition that strictly generalizes~\eqref{eq:L0L1-smoothness-hessian}, at the cost of rates that remain suboptimal even at \(\alpha=1\), while \citet{tyurin2025unifiedtheorygradientdescent} attains the best known nonconvex and convex rates for \(\alpha\le1\) through an integral-based stepsize rule, with no guarantee in the strongly convex case (\Cref{tab:L0L1-complexity-alpha}).
Both require knowledge of \(\alpha\), \(L_0\) and \(L_1\).

Hessian approximations have also been considered: the Gradient-Normalized Smoothness framework of \citet{semenov2025gradientnormalized} attains the best known nonconvex and convex rates for \((1,L_0,L_1)\)-smooth problems, but still selects its stepsize by grid search or adaptive tuning over \((L_0,L_1)\).

A few works dispense with prior knowledge of \(L_0\) and \(L_1\), at the cost of either a worse rate or another requirement.
In the nonconvex case, and only for \(\alpha=1\), the backtracking gradient method of \citet{hubler2024parameteragnostic} is not optimal: its rate involves a uniform bound on the gradient norm established in the analysis (\Cref{tab:L0L1-complexity}).
Recently, \citet{jerad2026fast} proposed a modified variant of the adaptive regularization method AR\(p\), in which the gradient is evaluated at every trial point and only exact derivatives are allowed.
Similarly to the trust-region method, their framework does not require prior knowledge of \(L_0\) or \(L_1\), but yields a worse complexity bound than ours when \(p=1\)~(\Cref{tab:L0L1-complexity}).
In the convex case, \citet{vankov2024optimizing} also analyze two methods that retain the best known rate without \(L_0\) and \(L_1\), but substitute another quantity for them: a Polyak stepsize, which requires a lower bound on the optimal value, and a normalized gradient method, which requires an estimate of \(R_0=\|x_0-x^\star\|\).
\citet{takezawa2024polyak} propose an inexact Polyak stepsize that requires no such quantity, but their rate deteriorates from the optimal convex one to the nonconvex one, with an additional term measuring the quality of the lower bound used (\Cref{tab:L0L1-complexity}).
Finally, the convex bound of \citet{jerad2026fast} also contains an additional logarithmic factor in \(\epsilon^{-1}\) (\Cref{tab:L0L1-complexity}).
The best known rates are thus attained only by substituting some other problem-dependent quantity for \(L_0\) and \(L_1\); attaining them with none at all, simultaneously in the nonconvex, convex and strongly convex settings, is our main contribution.

We summarize the best known complexity bounds in \Cref{tab:L0L1-complexity,tab:L0L1-complexity-alpha}.

\begin{table}[ht]%
  \centering
  \caption{
  Iteration complexity under \((1,L_0,L_1)\)-smoothness.
  Here \(\delta_0 = f(x_0)-f_{\rm low}\), \(R_0=\|x_0-x^\star\|\), where \(x^\star\) is a minimizer of \(f\), \(R \ge R_0\) is such that \(\|x - x^\star\| \le R\) for all \(x\) with \(f(x) \le f(x_0)\) (see Assumption~\ref{asm:convex}), \(M\) is an upper bound on the gradient norm, \(L\) is the Lipschitz constant of the gradient, \(\mu>0\) denotes the strong convexity constant, \(\sigma^2=f_{\rm low}-l^\star\ge0\), where \(l^\star\) is a lower bound on the optimal value \(f_{\rm low}\) and \(\epsilon>0\) is the desired accuracy.}%
  \label{tab:L0L1-complexity}
  \renewcommand{\arraystretch}{1.25}
  \setlength{\tabcolsep}{4pt}
  \scriptsize
  \begin{adjustbox}{max width=\textwidth}%
    \begin{tabular}{@{} l p{3.2cm} c c c c @{\hspace{30pt}} p{5.2cm} @{}}%
      \toprule
      \multirow{2}{*}{Class}
             &
      \multirow{2}{*}{Method}
             &
      \multicolumn{4}{c}{\hspace{-29pt}Parameters required}
             &
      \multirow{2}{*}{Complexity bound}
      \\
      \cmidrule(l{2pt}r{25pt}){3-6}
             &
             &
      \(L_0\)
             &
      \(L_1\)
             &
      \(f_{\rm low}\)
             &
      \(R_0\)
             &
      \\
      \midrule

      \multirow{4}{*}{\makecell[l]{Non-     \\
          convex}}
             &
      \citep[Theorem 6]{zhang2019gradient}
             &
      \cmark & \cmark & \xmark & \xmark
             &
      \(\displaystyle
      \delta_0
      \left(
      L_0 + L_1 M
      \right)
      \epsilon^{-2}
      \)
      \\
             &
      \citep[Theorem 3.1]{zhang2020improved}
             &
      \cmark & \cmark & \xmark & \xmark
             &
      \(\displaystyle
      \delta_0
      \left(
      L_0\epsilon^{-2}
      +
      L_1^2/L_0
      \right)
      \)
      \\
             &
      \citep[Theorem 3.1]{koloskova2023revisiting}, \citep[Theorem 2.1]{vankov2024optimizing}
             &
      \cmark & \cmark & \xmark & \xmark
             &
      \(\displaystyle
      \delta_0
      \left(
      L_0\epsilon^{-2}
      +
      L_1\epsilon^{-1}
      \right)
      \)
      \\

             &
      \citep[Theorem 4]{hubler2024parameteragnostic}
             &
      \xmark & \xmark & \xmark & \xmark
             &
      \(\displaystyle
      \delta_0 (L_0+\delta_0 L_1^2)\epsilon^{-2}
      \)
      \\
             &
      \citep[Theorem 3.1]{jerad2026fast}
             &
      \xmark & \xmark & \xmark & \xmark
             &
      \(\displaystyle
      \delta_0 (L_0+ L_1) \log(\epsilon^{-1}) \, \epsilon^{-2}
      \)
      \\
             &
      \Cref{thm:complexity:S}
             &
      \xmark & \xmark & \xmark & \xmark
             &
      \(\displaystyle
      \delta_0
      \left(
      L_0\epsilon^{-2}
      +
      L_1\epsilon^{-1}
      \right)
      \)
      \\

      \midrule

      \multirow{6}{*}{Convex}
             &
      \citep[Theorem 2.3]{koloskova2023revisiting}
             &
      \cmark & \cmark & \xmark & \xmark
             &
      \(\displaystyle
      L_0R_0^2\epsilon^{-1}
      +
      L_1R_0^2 \sqrt{L \epsilon^{-1}}
      \)
      \\
             &
      \citep[Theorem 3.3]{gorbunov2025methods}
             &
      \cmark & \cmark & \xmark & \xmark
             &
      \(\displaystyle
      L_0R_0^2\epsilon^{-1}
      +
      (L_1R_0)^2
      \)
      \\
             &
      \citep[Theorem 5.2]{gorbunov2025methods}
             &
      \cmark & \cmark & \xmark & \xmark
             &
      \(\displaystyle \sqrt{
        L_0 R_0^2 (1 + L_1 R_0^2 \exp(L_1 R_0))\epsilon^{-1}}
      \)
      \\
             &
      \citep[Theorem 3.2]{vankov2024optimizing}
             &
      \cmark & \cmark & \xmark & \xmark
             &
      \(\displaystyle
      L_0R_0^2\epsilon^{-1}
      +
      L_1R_0\log(\delta_0/\epsilon)
      \)
      \\
             &
      \citep[Theorem 6.2]{vankov2024optimizing}
             &
      \cmark & \cmark & \xmark & \xmark
             &
      \(\displaystyle \sqrt{
        L_0 R_0^2\epsilon^{-1}} +(L_1 R_0)^{2/3} \log(\delta_0/\epsilon)
      \)
      \\
             &
      \citep[Theorem 4.1]{gorbunov2025methods}, \citep[Theorem 5.1]{vankov2024optimizing}
             &
      \xmark & \xmark & \cmark & \xmark
             &
      \(\displaystyle
      L_0R_0^2\epsilon^{-1}
      +
      (L_1R_0)^2
      \)
      \\

             &
      \citep[Theorem 4.1]{vankov2024optimizing}
             &
      \xmark & \xmark & \xmark & \cmark
             &
      \(\displaystyle
      L_0R_0^2\epsilon^{-1}
      +
      (L_1 R_0)^2
      \)
      \\
             &
      \citep[Theorem 5]{takezawa2024polyak}
             &
      \xmark & \xmark & \xmark & \xmark
             &
      \(\displaystyle
      (L_0R_0^2+\sigma^2)\epsilon^{-2} + (L L_1^2 R_0^4 + L_1^2 L \sigma^4/L_0^2)\epsilon^{-1}
      \)
      \\
             &
      \citep[Theorem 3.2]{jerad2026fast}
             &
      \xmark & \xmark & \xmark & \xmark
             &
      \(\displaystyle
      \delta_0 (L_0+ L_1) \log(\epsilon^{-1}) \, \epsilon^{-1}
      \)
      \\
             &
      \Cref{thm:complexity:S-convex-2}
             &
      \xmark & \xmark & \xmark & \xmark
             &
      \(\displaystyle
      L_0R^2\epsilon^{-1}
      +
      L_1R\log(\delta_0/\epsilon)
      \)
      \\

      \midrule

      \multirow{3}{*}{\makecell[l]{Strongly \\
          convex}}
             &
      \citep[Theorem 2.5]{koloskova2023revisiting}
             &
      \cmark & \cmark & \xmark & \xmark
             &
      \(\displaystyle
      L_0/\mu \log(R_0^2/\epsilon)+ L_0 R_0 \min\{\sqrt{L/\mu}, LR_0\}
      \)
      \\
             &
      \citep[Theorem D.3]{gorbunov2025methods}
             &
      \cmark & \cmark & \xmark & \xmark
             &
      \(\displaystyle
      L_0/\mu \log(R_0^2/\epsilon)+(L_1R_0)^2
      \)
      \\
             &
      \Cref{thm:complexity:S-str-convex-2}
             &
      \xmark & \xmark & \xmark & \xmark
             &
      \(\displaystyle
      L_0/\mu \log(\delta_0/\epsilon)
      +
      L_1 (\delta_0/\mu)^{1/2}
      \)
      \\

      \bottomrule
    \end{tabular}
  \end{adjustbox}
\end{table}

\begin{table}[ht]%
  \centering
  \caption{
    Iteration complexity under \((\alpha,L_0,L_1)\)-smoothness, where \(\alpha \in (0,1)\).
    The conventions are those of \Cref{tab:L0L1-complexity}.}%
  \label{tab:L0L1-complexity-alpha}
  \renewcommand{\arraystretch}{1.25}
  \setlength{\tabcolsep}{4pt}
  \scriptsize
  \begin{adjustbox}{max width=\textwidth}%
    \begin{tabular}{@{} l p{3.2cm} c c c @{\hspace{30pt}} p{5.2cm} @{}}%
      \toprule
      \multirow{2}{*}{Class}
             &
      \multirow{2}{*}{Method}
             &
      \multicolumn{3}{c}{\hspace{-20pt}Parameters required}
             &
      \multirow{2}{*}{Complexity bound}
      \\
      \cmidrule(l{1pt}r{20pt}){3-5}
             &
             &
      \(L_0\)
             &
      \(L_1\)
             &
      \(\alpha\)
             &
      \\
      \midrule

      \multirow{4}{*}{\makecell[l]{Non-     \\
          convex}}
             &
      \citep[Theorem 2]{chen-zhou-liang-lu-2023}
             &
      \cmark & \cmark & \cmark
             &
      \(\displaystyle
      \delta_0(
      L_0(2^{\frac{\alpha^2}{1-\alpha}}+1) \epsilon^{-2}
      +
      L_1\,2^{\frac{\alpha^2}{1-\alpha}}\,3^\alpha
      \,\epsilon^{\alpha - 2})
      \)
      \\
             &
      \citep[Theorem 5.2]{li2023convex}
             &
      \cmark & \cmark & \cmark
             &
      \(\displaystyle
      \delta_0(L_0 + L_1 \|\nabla f(x_0)\|^{\alpha})\epsilon^{-2}
      \)
      \\
             &
      \citep[Theorem 5.1]{tyurin2025unifiedtheorygradientdescent}
             &
      \cmark & \cmark & \cmark
             &
      \(\displaystyle
      \delta_0
      \left(
      L_0\epsilon^{-2}
      +
      L_1\epsilon^{\alpha-2}
      \right)
      \)
      \\
             &
      \Cref{thm:complexity:S}
             &
      \xmark & \xmark & \xmark
             &
      \(\displaystyle
      \delta_0
      \left(
      L_0\epsilon^{-2}
      +
      L_1\epsilon^{\alpha-2}
      \right)
      \)
      \\

      \midrule

      \multirow{3}{*}{Convex}
             &
      \citep[Theorem 4.2]{li2023convex}
             &
      \cmark & \cmark & \cmark
             &
      \(\displaystyle
      (L_0R_0^2 + L_1R_0^2\|\nabla f(x_0)\|^{\alpha})\epsilon^{-1}
      \)
      \\
             &
      \citep[Theorem 8.1]{tyurin2025unifiedtheorygradientdescent}
             &
      \cmark & \cmark & \cmark
             &
      \(\displaystyle L_0 R_0^2\epsilon^{-1} + L_1^{\frac{2}{2 - \alpha}} R_0^2 \epsilon^{\frac{-2(1 - \alpha)}{2 - \alpha}}\)
      \\
             &
      \Cref{thm:complexity:S-convex-2}
             &
      \xmark & \xmark & \xmark
             &
      \(\displaystyle
      L_0 R^2 \epsilon^{-1} + L_1 R^{2-\alpha} (\epsilon^{\alpha-1} - \delta_0^{\alpha-1}) \,/ \,(1-\alpha)
      \)
      \\

      \midrule

      \multirow{2}{*}{\makecell[l]{Strongly \\
          convex}}
             &
      \citep[Theorem 4.3]{li2023convex}
             &
      \cmark & \cmark & \cmark
             &
      \(\displaystyle
      (L_0 + L_1 \|\nabla f(x_0)\|^{\alpha})/\mu \, \log(R_0^2/\epsilon)
      \)
      \\
             &
      \Cref{thm:complexity:S-str-convex-2}
             &
      \xmark & \xmark & \xmark
             &
      \(\displaystyle
      L_0/\mu \log(\delta_0/\epsilon)
      +
      L_1 \mu^{\alpha/2 - 1}(\delta_0^{\alpha/2} - \epsilon^{\alpha/2}) / \alpha
      \)
      \\

      \bottomrule
    \end{tabular}
  \end{adjustbox}
\end{table}

\subsection*{Notation}

For a finite set $\mathcal{A}$, $|\mathcal{A}|$ denotes its cardinality.
For a vector $x$ and matrix $B$, $\|x\|$ and $\|B\|$ denote the Euclidean norm and its induced matrix norm, respectively.
We denote by \(\mathbb S^{n-1}\) the unit sphere in \(\mathbb{R}^n\).
In what follows, \(L\)-smoothness means that \(\nabla f\) is Lipschitz continuous with constant \(L\).
Recall that a polynomial \(P:\mathbb{R}^n\to\mathbb{R}^m\) is said to be homogeneous of degree \(p\) if \(P(tx)=t^{p} P(x)\) for every \(x\in\mathbb{R}^n\) and every \(t\in\mathbb{R}\).
For \(x\in\mathbb R^n\), let \(x^{\otimes p}\) denote the \(p\)-fold tensor product of \(x\), i.e., the order-\(p\) tensor with entries \((x^{\otimes p})_{i_1,\ldots,i_p} \defeq x_{i_1}\cdots x_{i_p}\), where \(i_1,\ldots,i_p\) range over \(\{1,\ldots,n\}\).
Hence, for any order-\(p\) tensor \(A\),
\[
  \langle A,x^{\otimes p}\rangle = \sum_{i_1,\ldots,i_p=1}^n A_{i_1,\ldots,i_p}x_{i_1}\cdots x_{i_p}.
\]
Finally, we use the convention that \(0^0=1\).

\section{Preliminaries}%
\label{sec:preliminaries}

We make the following standard assumptions on the problem.

\begin{problemassumption}%
  \label{asm:problem-obj}
  The objective function \(f\) is continuously differentiable and bounded below.
  We denote its infimum by \(f_{\textup{low}}\).
\end{problemassumption}

We recall the definition of \((\alpha,L_0,L_1)\)-smoothness.
The case \(\alpha=1\)~\eqref{eq:L0L1-smoothness-hessian} was introduced by \citet{zhang2019gradient} for twice continuously differentiable functions; \citet{zhang2020improved} observed that the integral form below is meaningful for merely continuously differentiable functions, and \citet{chen-zhou-liang-lu-2023} extended it to general \(\alpha\in[0,1]\).

\begin{definition}[\citep{zhang2019gradient,zhang2020improved,chen-zhou-liang-lu-2023}]%
  \label{def:smooth-integral}
  Let \(f \colon \R^n \to \R\) be continuously differentiable, \(\alpha \in [0,1]\), and \(L_0 > 0\) or \(L_1 > 0\).
  We say that \(f\) is \emph{\((\alpha,L_0,L_1)\)-smooth} if, for all \(x,y \in \R^n\),
  \begin{equation}
    \label{def-generalized-smooth-integral}
    \| \nabla f(x) - \nabla f(y) \|
    \leq
    (
    L_0 + L_1 \int_{0}^{1} \| \nabla f(x + t(y-x))\|^\alpha \, dt
    )
    \|x - y\|.
  \end{equation}
\end{definition}

\Cref{def:smooth-integral} yields the following equivalent form as proved by \citet[Lemma~\(A.1\)]{chen-zhou-liang-lu-2023}.

\begin{equation}
  \label{def-generalized-smooth}
  \| \nabla f(x) - \nabla f(y) \|
  \leq
  (
  L_0
  +
  L_1 \sup_{u \in [x,y]} \| \nabla f(u) \|^\alpha
  )
  \|x - y\|,
  \qquad \forall x,y \in \R^n.
\end{equation}

We make the following assumption on the objective function \(f\).

\begin{problemassumption}%
  \label{asm:smooth}
  We assume that \(f\) is \((\alpha,L_0,L_1)\)-smooth in the sense of \Cref{def:smooth-integral}, for some \(\alpha \in [0,1]\) and \(L_0,L_1 \geq 0\).
\end{problemassumption}

When \(f \in \mathcal{C}^2\), \Cref{asm:smooth} is equivalent to~\eqref{eq:L0L1-smoothness-hessian}.

\subsection{Examples}

Several classes of functions satisfy \Cref{asm:smooth}.
For instance, a univariate polynomial of degree at least three is not globally \(L\)-smooth, yet it satisfies the \((1,L_0,L_1)\)-smoothness condition~\citep{zhang2019gradient}.
Functions such as \(f(x)=\|x\|^{2n}\), \(f(x)=\exp(a^T x)\), and the logistic loss also satisfy the \((1,L_0,L_1)\)-smoothness condition \citep{gorbunov2025methods}.
A class of machine learning problems is also known to satisfy \((1,L_0,L_1)\)-smoothness, including distributionally robust optimization \citep{chen-zhou-liang-lu-2023}, symmetric and asymmetric matrix factorization \citep{bodard2026escaping}, and the training of neural networks \citep{zhang2019gradient}.
By contrast, to the best of our knowledge, phase retrieval \citep{chen-zhou-liang-lu-2023}, which satisfies \((\tfrac{2}{3},L_0,L_1)\)-smoothness, is the only explicit example previously exhibited for \(0 < \alpha < 1\).

In the following, we introduce a generic class of \((\alpha,L_0,L_1)\)-smooth functions, which we call \emph{polynomial inverse problems}, and which includes phase retrieval \citep{chen-zhou-liang-lu-2023} as a special case.

\begin{theorem}[Polynomial inverse problems]%
  \label{thm:polynomial-inverse-problem}
  Let \(\Phi:\mathbb R^n\rightarrow\mathbb R^m\) be a polynomial map of degree \(p \ge 1\), let \(y\in\mathbb R^m\) be fixed data, and consider \(f(x) = \frac12\|\Phi(x)-y\|^2\).
  Write \(\Phi = \Phi_p+\Phi_{p-1}+\cdots+\Phi_0, \) where $\Phi_k$ is homogeneous of degree $k$.

  Define the leading homogeneous objective
  \(
  f_{2p}(x)
  =
  \frac12\|\Phi_p(x)\|^2.
  \)

  Assume the following noncriticality condition on the unit sphere,
  \[
    \nabla f_{2p}(u)\neq0,
    \qquad
    \forall u\in\mathbb S^{n-1}.
  \]

  Then there exist constants $L_0,L_1>0$ such that
  \begin{equation}
    \label{eq:polynomial-inverse-problem}
    \|\nabla^2f(x)\|
    \le
    L_0+
    L_1
    \|\nabla f(x)\|^\alpha,
    \quad
    \alpha
    =
    \frac{2p-2}{2p-1},
    \quad
    \forall x\in\mathbb R^n.
  \end{equation}
  For \(p=1\), \(\alpha=0\), which is consistent with \(f\) being a convex quadratic, hence \(L\)-smooth.
\end{theorem}

\begin{proof}
  Since $\Phi$ has degree $p$, $f(x)=\frac12\|\Phi(x)-y\|^2$ is a polynomial of degree $2p$ with leading homogeneous part $f_{2p}(x)=\frac12\|\Phi_p(x)\|^2$.
  Hence,
  \[
    \nabla f_{2p}(x) = O(\|x\|^{2p-1}) \quad\text{and}\quad \nabla^2 f_{2p}(x) = O(\|x\|^{2p-2}),
  \]
  and therefore,
  \[
    \nabla f(x)=\nabla f_{2p}(x)+O(\|x\|^{2p-2})
    \quad\text{and}\quad
    \nabla^2f(x)=\nabla^2f_{2p}(x)+O(\|x\|^{2p-3}),
  \]
  thus there exist \(C_0\), \(C_1>0\) such that
  \begin{equation}%
    \label{eq:hessian-bound-polynomial-inverse-problem}
    \|\nabla^2 f(x)\| \le C_0 + C_1 \|x\|^{2p-2}, \qquad \forall x\in\mathbb R^n.
  \end{equation}
  Moreover, since $f_{2p}$ is homogeneous of degree $2p$,
  \[
    \|\nabla f_{2p}(x)\|
    =
    \|x\|^{2p-1}
    \|
    \nabla f_{2p}\!\left(\frac{x}{\|x\|}\right)
    \|,
  \]
  Since \(\mathbb S^{n-1}\) is compact and \(\nabla f_{2p}\) is continuous, the noncriticality assumption ensures that \(\mu\defeq\min_{u\in\mathbb S^{n-1}}\|\nabla f_{2p}(u)\|>0\), hence \(\|\nabla f_{2p}(x)\|\ge\mu\|x\|^{2p-1}\) for every \(x\).

  Because \(\nabla f(x) = \nabla f_{2p}(x) + O(\|x\|^{2p-2})\), there exist \(C_2\), \(C_3>0\) such that
  \[
    \|\nabla f(x)\|
    \ge
    \|\nabla f_{2p}(x)\| - C_2 \|x\|^{2p-2} - C_3
    \ge
    \mu\|x\|^{2p-1} - C_2 \|x\|^{2p-2} - C_3
    \ge
    \frac{\mu}{2}\|x\|^{2p-1},
  \]
  where the last inequality holds for all sufficiently large \(\|x\|\) outside a compact set \(K\subset\mathbb R^n\).
  Substituting this estimate into the Hessian bound~\eqref{eq:hessian-bound-polynomial-inverse-problem} gives the desired inequality outside \(K\).
  Enlarging \(L_0\), if necessary, extends the inequality to all \(x\in\mathbb R^n\).
\end{proof}


We conclude this section with a few examples of polynomial inverse problems.

\begin{example}%
  Let \(\Phi(x) = \bigl(\langle A_1,x^{\otimes p}\rangle,\ldots, \langle A_m,x^{\otimes p}\rangle \bigr)\), where \(A_1,\ldots,A_m\) are symmetric tensors of order \(p\), and consider \( f(x)= \frac12 \sum_{i=1}^m \bigl( \langle A_i,x^{\otimes p}\rangle-y_i \bigr)^2\).
  The noncriticality assumption holds whenever
  \begin{equation}%
    \label{eq:tensor-noncriticality}
    \sum_{i=1}^m
    \langle A_i,u^{\otimes p}\rangle^2
    >
    0,
    \qquad
    \forall\,u\in\mathbb S^{n-1}.
  \end{equation}
  Then, by \Cref{thm:polynomial-inverse-problem},
  \[
    \|\nabla^2f(x)\|
    \le
    L_0+
    L_1
    \|\nabla f(x)\|^\alpha,
    \qquad
    \alpha=\frac{2p-2}{2p-1}.
  \]
\end{example}

The next two examples are the special cases \(A_i=a_i a_i^T\) which corresponds to phase retrieval \citep{chen-zhou-liang-lu-2023}, and \(A_i=Q_i\) symmetric positive semidefinite matrices, which corresponds to the case \(p=2\).

\begin{example}[Phase retrieval {\normalfont\citep{chen-zhou-liang-lu-2023}}]%
  This corresponds to the tensor sensing problem with \(A_i = a_i^{\otimes 2} = a_i a_i^T\) for some vectors \(a_1,\ldots,a_m \in \mathbb R^n\).
  The noncriticality assumption holds whenever
  \[
    \operatorname{span}\{a_1,\ldots,a_m\}
    =
    \mathbb R^n,
  \]
  Therefore,
  \[
    \|\nabla^2f(x)\|
    \le
    L_0+
    L_1
    \|\nabla f(x)\|^{2/3},
  \]
\end{example}

\begin{example}%
  We now consider the case \(p=2\) with \(A_i=Q_i\) symmetric positive semidefinite matrices, so that \(\Phi(x) = (x^T Q_1 x,\ldots,x^T Q_m x)\).
  The noncriticality assumption holds whenever
  \[
    \bigcap_{i=1}^m
    \ker(Q_i)
    =
    \{0\}.
  \]
  Hence, according to \Cref{thm:polynomial-inverse-problem},
  \[
    \|\nabla^2f(x)\|
    \le
    L_0+
    L_1
    \|\nabla f(x)\|^{2/3}.
  \]
\end{example}

\begin{proof}
  Here \(\langle Q_i,u^{\otimes2}\rangle = u^\top Q_i u\), so~\eqref{eq:tensor-noncriticality} reads \(\sum_{i=1}^m(u^\top Q_i u)^2>0\).
  If this failed for some \(u\in\mathbb S^{n-1}\), then \(u^\top Q_i u=0\) for every \(i\), and \(Q_i\succeq0\) would give \(Q_i u=0\) for every \(i\), that is, \(u\in\bigcap_{i=1}^m\ker(Q_i)=\{0\}\), contradicting \(\|u\|=1\).
\end{proof}

\subsection{Properties of \((\alpha, L_0, L_1)\)-smooth functions}%
\label{sec:properties}

Before presenting the trust-region algorithm, we need the following characterization of \((\alpha, L_0, L_1)\)-smoothness that will be useful in the analysis of the algorithm.
The proof is provided in \Cref{sec:proofs}.

\begin{theoremE}[][end]%
  \label{lem:decrease-L0-L1}
  Let \Cref{asm:problem-obj,asm:smooth} hold.
  Then, for all \(x, s\in\R^n\), such that \(\nabla f(x)\neq0\), we have
  \begin{equation}
    \label{eq:gradient-bound}
    \|\nabla f(x)\|
    \le
    \sqrt{8L_0\bigl(f(x)-f_{\textup{low}}\bigr)}
    +
    \left(
    2^{\alpha+3}L_1\bigl(f(x)-f_{\textup{low}}\bigr)
    \right)^{\frac{1}{2-\alpha}},
  \end{equation}
  \begin{equation}
    \label{eq:gradient-difference}
    \|\nabla f(x + s)-\nabla f(x)\|
    \leq
    \bigl(L_0+L_1\|\nabla f(x)\|^\alpha\bigr)
    \galpha(\|s\|,\|\nabla f(x)\|)\|s\|,
  \end{equation}
  and
  \begin{equation}
    \label{eq:function-difference}
    |f(x + s)-f(x)- \nabla f(x)^T s|
    \leq
    \tfrac12
    \bigl(L_0+L_1\|\nabla f(x)\|^\alpha\bigr)
    \galpha(\|s\|,\|\nabla f(x)\|)\|s\|^2,
  \end{equation}
  where \(g :[0,\infty)\times(0,\infty) \to (0,\infty)\) is defined by
  \begin{equation}
    \label{eq:def-g-alpha}
    \galpha(a,b)
    \defeq
    \begin{cases}
      \left(
      1+
      \dfrac{(1-\alpha)L_1a^\alpha}
        {(L_0+L_1b^\alpha)^{1-\alpha}}
      \right)^{\frac{1}{1-\alpha}},
       &
      \text{if }\alpha\in[0,1) \text{ and }L_0+L_1b^\alpha>0,
      \\
      \exp(L_1a),
       & \text{if }\alpha=1.
    \end{cases}
  \end{equation}
\end{theoremE}

\begin{proofEE}
  We split the proof into two parts.
  \paragraph{Proof of~\eqref{eq:gradient-bound}}
  Let \(\nu\defeq\|\nabla f(x)\|\), \(d\defeq-\nabla f(x)/\nu\), \(q\defeq L_0+2^\alpha L_1\nu^\alpha > 0\) and \(r\defeq\nu/(2q)\).
  First claim is that \(\|\nabla f(x+td)\|\leq2\nu\) for all \(t\in[0,r]\).
  Indeed, if \(t_0\leq r\) were the first time such that \(\|\nabla f(x+t_0d)\|=2\nu\), then~\eqref{def-generalized-smooth}, applied on the segment \([x,x+t_0d]\) on which \(\|\nabla f\|\le2\nu\) by minimality of \(t_0\), gives \(\|\nabla f(x+t_0d)-\nabla f(x)\| \leq (L_0+L_1(2\nu)^\alpha)t_0 \leq q r=\nu/2\), which would imply \(\|\nabla f(x+t_0d)\|\leq3\nu/2\), a contradiction.
  Thus, \(\|\nabla f(x+td)-\nabla f(x)\|\leq q t\) for every \(t\in[0,r]\).
  Since \(\langle\nabla f(x),d\rangle=-\nu\),
  \[
    f(x+rd)-f(x)
    \leq
    \int_0^r(-\nu+q t)\,dt
    =
    -\frac{3\nu^2}{8q}
    \leq
    -\frac{\nu^2}{4q}.
  \]
  As \(f(x+rd)\geq f_{\textup{low}}\), denoting \(\delta_x \defeq f(x)-f_{\textup{low}}\), we obtain \(\nu^2 \leq 4 L_0 \delta_x+ 2^{\alpha+2}L_1\delta_x\,\nu^\alpha\).
  We use the implication \(z^2\le a+bz^\alpha \Rightarrow z\le\sqrt{2a}+(2b)^{1/(2-\alpha)}\), valid for \(z,a,b\ge0\) and \(\alpha\in[0,2)\).
  Indeed, either \(a\ge bz^\alpha\), in which case \(z^2\le2a\) and \(z\le\sqrt{2a}\), or \(a<bz^\alpha\), in which case \(z^2\le2bz^\alpha\), that is, \(z^{2-\alpha}\le2b\) and \(z\le(2b)^{1/(2-\alpha)}\).
  In both cases \(z\) is at most the sum of the two terms.
  Applying it here, we obtain
  \[
    \|\nabla f(x)\|
    \leq
    \sqrt{8L_0\delta_x}
    +
    \bigl(2^{\alpha+3}L_1\delta_x\bigr)^{\frac1{2-\alpha}}.
  \]

  \paragraph{Proof of~\eqref{eq:gradient-difference} and~\eqref{eq:function-difference}}
  For \(\alpha=1\), the result follows from \citep[Lemmas~2.5 and~2.6]{vankov2024optimizing}.
  Assume now that \(\alpha\in[0,1)\).
  If \(s = 0\), the result is trivial.
  For \(t\in[0,1]\), define \(\chi(t) \defeq \|s\|\int_0^t \bigl(L_0+L_1\|\nabla f(x+\tau s)\|^\alpha\bigr)\,d\tau\).
  By \((\alpha,L_0,L_1)\)-smoothness applied to \(x\) and \(x+ts\),
  \begin{align*}
    \|\nabla f(x+ts)-\nabla f(x)\|
     & \leq
    \|ts\|\int_0^1
    \bigl(L_0+L_1\|\nabla f(x+\tau ts)\|^\alpha\bigr)\,d\tau
    \\
     & =
    \|s\|\int_0^t
    \bigl(L_0+L_1\|\nabla f(x+\tau s)\|^\alpha\bigr)\,d\tau
    =
    \chi(t),
  \end{align*}
  where the equality follows from the change of variable \(\tau'\defeq t\tau\).
  Since \(\chi(0)=0\), we have
  \begin{align*}
    \chi'(t)
     & =
    \|s\|
    \bigl(L_0+L_1\|\nabla f(x+ts)\|^\alpha\bigr)
    \leq
    \|s\|
    \bigl(
    L_0+L_1(\|\nabla f(x)\|+\chi(t))^\alpha
    \bigr)
    \\
     & \leq
    \|s\|
    \bigl(
    L_0+L_1\|\nabla f(x)\|^\alpha+L_1\chi(t)^\alpha
    \bigr)
    =
    \lambda+\beta\chi(t)^\alpha,
  \end{align*}
  where \(\lambda\defeq \bigl(L_0+L_1\|\nabla f(x)\|^\alpha\bigr)\|s\|\) and \(\beta\defeq L_1\|s\|\).
  Here, we used \((a+b)^\alpha\leq a^\alpha+b^\alpha\), valid for \(a,b\geq0\) and \(\alpha\in[0,1]\).
  Integrating the preceding inequality gives \(\chi(t)\leq \lambda t+\beta\int_0^t\chi(\tau)^\alpha\,d\tau\).
  For any \(\varrho>0\), define \(\zeta_\varrho(t)\defeq \lambda+\varrho+ \beta\int_0^t\chi(\tau)^\alpha\,d\tau\).
  Since \(t\leq1\), we have \(\chi(t)\leq\zeta_\varrho(t)\), and hence \(\zeta_\varrho'(t) =\beta\chi(t)^\alpha \leq\beta\zeta_\varrho(t)^\alpha\).
  Because \(\zeta_\varrho(t)>0\),
  \[
    \frac{d}{dt}\zeta_\varrho(t)^{1-\alpha}
    =
    (1-\alpha)\zeta_\varrho(t)^{-\alpha}
    \zeta_\varrho'(t)
    \leq
    (1-\alpha)\beta.
  \]
  Integrating from \(0\) to \(1\) yields \(\zeta_\varrho(1)^{1-\alpha} \leq (\lambda+\varrho)^{1-\alpha} +(1-\alpha)\beta\).
  Since \(\chi(1)\leq\zeta_\varrho(1)\), letting \(\varrho\downarrow0\) gives
  \[
    \chi(1)
    \leq
    \left(
    \lambda^{1-\alpha}+(1-\alpha)\beta
    \right)^{\frac{1}{1-\alpha}}.
  \]
  Substituting the definitions of \(\lambda\) and \(\beta\), we obtain
  \begin{align*}
    \chi(1)
     & \leq
    \|s\|
    \left(
    \bigl(L_0+L_1\|\nabla f(x)\|^\alpha\bigr)^{1-\alpha}
    +(1-\alpha)L_1\|s\|^\alpha
    \right)^{\frac{1}{1-\alpha}}
    \\
     & =
    \bigl(L_0+L_1\|\nabla f(x)\|^\alpha\bigr)
    \galpha(\|s\|,\|\nabla f(x)\|)\|s\|,
  \end{align*}
  where the final identity follows from~\eqref{eq:def-g-alpha} and \(L_0+L_1\|\nabla f(x)\|^\alpha>0\).
  This proves~\eqref{eq:gradient-difference}.
  Finally, the fundamental theorem of calculus gives
  \begin{align*}
    |f(x+s)-f(x)- \nabla f(x)^T s|
     & \leq
    \int_0^1
    \|\nabla f(x+ts)-\nabla f(x)\|\|s\|\,dt
    \\
     & \leq
    \bigl(L_0+L_1\|\nabla f(x)\|^\alpha\bigr)
    \galpha(\|s\|,\|\nabla f(x)\|)
    \|s\|^2\int_0^1t\,dt
    \\
     & =
    \tfrac12
    \bigl(L_0+L_1\|\nabla f(x)\|^\alpha\bigr)
    \galpha(\|s\|,\|\nabla f(x)\|)\|s\|^2,
  \end{align*}
  where we used~\eqref{eq:gradient-difference} with \(ts\) in place of \(s\) and the monotonicity of \(\galpha(\cdot,\|\nabla f(x)\|)\).
\end{proofEE}

When \(\alpha = 1\), the function \(\galpha\) only depends on \(a\) and not on \(b\), but we keep the notation \(\galpha(a,b)\) for consistency with the case \(\alpha \in [0,1)\).

The bounds in \Cref{lem:decrease-L0-L1} are tighter than those of \citet{chen-zhou-liang-lu-2023} under the \((\alpha,L_0,L_1)\)-smoothness condition for \(\alpha < 1\).
Moreover, as \(\alpha \to 1^-\), the function \(\galpha\) converges to \(\exp(L_1a)\), whereas the bound of \citet{chen-zhou-liang-lu-2023} diverges to \(+\infty\).

In our analysis, we will need the following properties.
The proof of these properties is straightforward and is omitted for brevity.

\begin{proposition}%
  \label{prop:galpha-properties}
  For \(\alpha \in [0,1]\), let \(\galpha\) be defined as in \Cref{lem:decrease-L0-L1}.
  Then,
  \begin{enumerate}
    \item\label{item:galpha-C1}
      \(\galpha\) is continuous on \([0,\infty)\times(0,\infty)\) and \(\mathcal{C}^1((0,\infty)^2)\).
    \item\label{item:galpha-monotone-a}
      For all fixed \(b > 0\), the function \(a \mapsto \galpha(a,b)\) is nondecreasing on \([0,\infty)\).
    \item\label{item:galpha-monotone-b}
      For all fixed \(a \ge 0\), the function \(b \mapsto \galpha(a,b)\) is nonincreasing on \((0,\infty)\).
    \item\label{item:galpha-limit-a0}
      For all \(\alpha>0\) and \(b > 0\), \( \lim_{a \to 0^+} \galpha(a,b) = 1\).
    \item\label{item:galpha-lower-bound}
      For all \(a \ge 0\) and  \(b > 0\), \(\galpha(a,b) \geq 1\).
  \end{enumerate}
\end{proposition}

In the next section, we present the trust-region algorithm that we analyze in this paper, and we state the main assumptions on the model Hessian.

\section{Trust-region algorithm}%
\label{sec:TR}

\subsection{Algorithm}%
\label{sec:TR-algorithm}

At each iteration \(k\), we compute a step \(s_k\) from the current iterate \(x_k\) as an (inexact) solution of the subproblem
\begin{align}%
  \label{subprob:TR}
   & \min_{s \in \R^{n}} m_k(s) \quad \text{s.t.} \ \|s\| \leq \Delta_k, \quad \text{where } m_k(s) \defeq f(x_k) + \nabla f(x_k)^{T} s + \tfrac{1}{2} s^{T} B_k s,
\end{align}
where \(B_k = B_k^T\) and \(\Delta_k > 0\) is the trust-region radius.

Only an approximate solution of~\eqref{subprob:TR} is required; the step \(s_k\) is required to achieve at least a fixed fraction of the Cauchy decrease.
The decrease required of \(s_k\) is stated as
\begin{equation}%
  \label{eq:suff-decrease}
  m_k(0) - m_k(s_k) \ge \kdc \|\nabla f(x_k)\| \min \left\{\frac{\|\nabla f(x_k)\|}{\|B_k\|}\, , \, \Delta_k \right\},
\end{equation}
where \(0 < \kdc \le \frac{1}{2}\), with the convention that \(\frac{\|\nabla f(x_k)\|}{\|B_k\|} = +\infty\) if \(B_k = 0\).
In particular, if \(s_k\) solves exactly the subproblem~\eqref{subprob:TR} then it satisfies~\eqref{eq:suff-decrease}.

Once a trial step \(s_k\) has been determined, the decrease in \(f\) at \(x_k + s_k\) is compared to the decrease predicted by the model.
If the actual and predicted decreases are in sufficient agreement, \(x_k + s_k\) becomes the new iterate, and \(\Delta_k\) is possibly increased.
If the model predicts the actual decrease poorly, the trial point is rejected and \(\Delta_k\) is reduced.
\Cref{alg:TR-GEN} states the entire procedure.

\begin{algorithm}[ht]%
  \caption{%
    \label{alg:TR-GEN}
    Classical trust-region algorithm \citet[Algorithm 2.3.1.]{cartis-gould-toint-2022}.}
  \begin{algorithmic}[1]%
    \State Choose \(x_0\in\R^n\), \(\Delta_0>0\), \(0<\eta_1\leq\eta_2<1\), and \(0<\gamma_1\leq\gamma_2<1\leq\gamma_3\).
    \For{\(k=0,1,\dots\)}
    \State%
    \label{alg:TR-GEN:Bk}
    Choose \(B_k=B_k^T\in\R^{n\times n}\).
    \State%
    \label{alg:TR-GEN:step-computation}
    Compute \(s_k\) satisfying~\eqref{eq:suff-decrease}.
    \State%
    \label{alg:TR-GEN:step-rhok}
    Set
    \[
      \rho_k\defeq
      \frac{f(x_k)-f(x_k+s_k)}
      {m_k(0)-m_k(s_k)}.
    \]
    \State%
    \label{alg:TR-GEN:step-accept}
    Set \(x_{k+1}=x_k+s_k\) if \(\rho_k\geq\eta_1\), and \(x_{k+1}=x_k\) otherwise.
    \State%
    \label{alg:TR-GEN:step-update}
    Choose
    \[
      \Delta_{k+1}\in
      \begin{cases}
        [\Delta_k,\gamma_3\Delta_k],
         & \rho_k\geq\eta_2,        \\
        [\gamma_2\Delta_k,\Delta_k],
         & \eta_1\leq\rho_k<\eta_2, \\
        [\gamma_1\Delta_k,\gamma_2\Delta_k],
         & \rho_k<\eta_1.
      \end{cases}
    \]
    \EndFor
  \end{algorithmic}
\end{algorithm}

In \Cref{alg:TR-GEN}, no upper bound is imposed on the trust-region radius \(\Delta_k\), which may therefore become arbitrarily large.

We will repeatedly use the notation
\begin{align*}%
  \mathcal{S}   & \defeq \{ i \in \N \mid \rho_i \geq \eta_1 \}               &  & \qquad \text{(all successful iterations)}
  \\
  \mathcal{S}_k & \defeq \{ i \in \mathcal{S} \mid i \le k \}                 &  & \qquad \text{(successful iterations until iteration \(k\))}
  \\
  \mathcal{U}_k & \defeq \{ i \in \N \mid i\not \in \mathcal{S}, \ i \le k \} &  & \qquad \text{(unsuccessful iterations until iteration \(k\))}.
\end{align*}

We recall the following standard relation between the numbers of successful and unsuccessful iterations~{\citep[Lemma~\(2.3.1\)]{cartis-gould-toint-2022}}.

\begin{lemma}%
  \label{lem:bound-unsuccessful-iterations}
  For all \(k \in \N^*\), \(\left|\mathcal{U}_{k-1}\right| \leq |\log_{\gamma_2} (\gamma_3)| \left|\mathcal{S}_{k-1}\right| +  \log_{\gamma_2}\left(\frac{\Delta_{k}}{\Delta_0}\right)\).
\end{lemma}

\Cref{lem:bound-unsuccessful-iterations} implies that the number of unsuccessful iterations is at most a constant times the number of successful iterations plus a logarithmic term that can be controlled using a lower bound on \(\Delta_k\).
Unless otherwise stated, we will not explicitly keep track of the number of unsuccessful iterations in our analysis and will only focus on the number of successful iterations and the lower bound on \(\Delta_k\).

\subsection{Model Hessian bounds}%
\label{subsec:model-hessian-bounds}

We make the following assumption.

\begin{modelassumption}%
  \label{asm:Bk-L0-L1}
  There exist \(b_0,b_1\ge0\) such that for all \(k\in\mathbb N\),
  \begin{equation}
    \label{def-generalized-smooth-Bk}
    \|B_k\|
    \le
    b_0+b_1\|\nabla f(x_k)\|^\alpha.
  \end{equation}
\end{modelassumption}

\Cref{asm:Bk-L0-L1} allows for three standard choices: \(B_k=0\) and \(B_k=\lambda I\) satisfy it with \(b_1=0\), the former yielding normalized-gradient-type trust-region steps and the latter clipped-gradient steps~\citep{zhang2019gradient}, while \(B_k=\nabla^2 f(x_k)\) satisfies it with \((b_0,b_1)=(L_0,L_1)\) whenever \(f\) is twice continuously differentiable, by~\eqref{eq:L0L1-smoothness-hessian}.

Moreover, \Cref{asm:Bk-L0-L1} is compatible with limited-memory quasi-Newton approximations as shown in the following proposition.

\begin{propositionE}[Limited-memory quasi-Newton updates][end]%
  \label{lem:limited-memory-qn-bounds-2}
  Let \(m\in\mathbb N\) be the memory size.
  At iteration \(k\), let \(B_k\) be obtained from an initial matrix \(B_{k,0}\)
  and a set of pairs \(\{(s_i,y_i)\}_{i\in\mathcal I_k}\) stored at
  \emph{successful} iterations, where
  \(\mathcal I_k=\{i_0,\ldots,i_{q_k-1}\}\subseteq\mathcal S\),
  \(q_k\defeq|\mathcal I_k|\le m\), \(s_i\neq0\), and
  \(y_i=\nabla f(x_i+s_i)-\nabla f(x_i)\).
  Assume that \(B_{k,0}\) satisfies \Cref{asm:Bk-L0-L1}.
  Then, for any of the following limited-memory quasi-Newton updates, \(B_k\) satisfies \Cref{asm:Bk-L0-L1}.

  \begin{enumerate}
    \renewcommand{\labelenumi}{\textup{(\roman{enumi})}}

    \item \textbf{Limited-memory BFGS \citep{burdakov-gong-zirkin-yuan-2017}.}
      Assume that \(B_{k,0}\succ0\) and every stored pair satisfies \(y_{i_j} \neq 0\) and \(s_{i_j}^T y_{i_j}\ge
      \omega_{\rm BFGS}\|s_{i_j}\|\|y_{i_j}\|\), where \(\omega_{\rm BFGS}\in(0,1)\).
      Set \(B_k=B_{k,q_k}\), where, for \(j=0,\ldots,q_k-1\),
      \begin{equation}
        \label{eq:LBFGS-update}
        B_{k,j+1}
        =
        B_{k,j}
        -
        \frac{B_{k,j}s_{i_j}s_{i_j}^T B_{k,j}}
             {s_{i_j}^T B_{k,j}s_{i_j}}
        +
        \frac{y_{i_j}y_{i_j}^T}{s_{i_j}^T y_{i_j}}.
      \end{equation}

    \item \textbf{Modified limited-memory BFGS.}
      Assume that \(B_{k,0}\succ0\) and every stored pair satisfies \(y_{i_j} \neq 0\) and \(s_{i_j}^T y_{i_j}\ge \bar\omega_{\rm BFGS}\|y_{i_j}\|^2\), where
      \(\bar\omega_{\rm BFGS}\in(0,1)\).
      Set \(B_k=B_{k,q_k}\) as in~\eqref{eq:LBFGS-update}.

    \item \textbf{Limited-memory SR1 \citep{burdakov-gong-zirkin-yuan-2017}.}
      Set \(z_{i_j}=y_{i_j}-B_{k,j}s_{i_j}\) and assume that every stored pair satisfies \(z_{i_j} \neq 0\) and \(|s_{i_j}^T z_{i_j}|\ge \omega_{\rm SR1}\|s_{i_j}\|\|z_{i_j}\|\), where \(\omega_{\rm SR1}\in(0,1)\).
      Set \(B_k=B_{k,q_k}\), where, for \(j=0,\ldots,q_k-1\),
      \begin{equation}
        \label{eq:LSR1-update}
        B_{k,j+1} = B_{k,j} + \frac{(y_{i_j}-B_{k,j}s_{i_j})(y_{i_j}-B_{k,j}s_{i_j})^T}{(y_{i_j}-B_{k,j}s_{i_j})^T s_{i_j}}.
      \end{equation}

    \item \textbf{Modified limited-memory SR1 \citep{bigeon2023framework,aravkin-baraldi-orban-2021}.}
      Set \(z_{i_j}=y_{i_j}-B_{k,j}s_{i_j}\) and assume that every stored pair satisfies \(z_{i_j} \neq 0\) and \(|s_{i_j}^T z_{i_j}|\ge \bar\omega_{\rm SR1}\|z_{i_j}\|^2\), where \(\bar\omega_{\rm SR1}\in(0,1)\).
      Set \(B_k=B_{k,q_k}\) as in~\eqref{eq:LSR1-update}.

    \item \textbf{Limited-memory PSB \citep{powell-2010}.}
      Set \(z_{i_j}=y_{i_j}-B_{k,j}s_{i_j}\) and
      \(B_k=B_{k,q_k}\), where, for \(j=0,\ldots,q_k-1\),
      \begin{equation}
        \label{eq:LPSB-update}
        B_{k,j+1}
        =
        B_{k,j}
        +
        \frac{z_{i_j}s_{i_j}^T+s_{i_j}z_{i_j}^T}{\|s_{i_j}\|^2}
        -
        \frac{z_{i_j}^T s_{i_j}}{\|s_{i_j}\|^4}
        s_{i_j}s_{i_j}^T.
      \end{equation}

    \item \textbf{Modified limited-memory PSB.}
      Assume that every stored pair satisfies \(\|s_{i_j}\|\ge \bar\omega_{\rm PSB}\|y_{i_j}-B_{k,j}s_{i_j}\|\), where \(\bar\omega_{\rm PSB}>0\).
      Set \(B_k=B_{k,q_k}\) as in~\eqref{eq:LPSB-update}.
  \end{enumerate}
\end{propositionE}

\begin{proofEE}
  We consider each update separately.

  \paragraph{(i) Limited-memory BFGS}
  Iterating over at most \(m\) stored pairs and applying \citet[Lemma~\(5\)]{burdakov-gong-zirkin-yuan-2017} and \Cref{lem:bounded-secant-slopes-L0-L1} shows that
  \[
  \|B_k\| \le \|B_{k,0}\| + \frac{m b_m}{\omega_{\rm BFGS}}
  \] 
  where \(b_m\) is the bound defined in \Cref{lem:bounded-secant-slopes-L0-L1}, and hence \(B_k\) satisfies \Cref{asm:Bk-L0-L1}.

  \paragraph{(ii) Modified limited-memory BFGS}
  Using the same argument, one get that,
  \[
    \|B_k\|
    \le
    \|B_{k,0}\|
    +
    \frac{m}{\bar \omega_{\rm BFGS}}
  \]
  and the conclusion follows directly.

  \paragraph{(iii) Limited-memory SR1}
  Let \(z_{i_j}\defeq y_{i_j}-B_{k,j}s_{i_j}\).
  Then,
  \[
    \|B_{k,j+1}\|
    \le
    \|B_{k,j}\|
    +
    \frac{\|z_{i_j}\|^2}{|z_{i_j}^T s_{i_j}|}
    \le
    \|B_{k,j}\|
    +
    \frac{\|z_{i_j}\|}
         {\omega_{\rm SR1}\|s_{i_j}\|}
    \le
    \left(1+\frac{1}{\omega_{\rm SR1}}\right)\|B_{k,j}\|
    +
    \frac{\|y_{i_j}\|}{\omega_{\rm SR1}\|s_{i_j}\|}.
  \]
  Iterating over at most \(m\) stored pairs and applying \Cref{lem:bounded-secant-slopes-L0-L1} shows that
  \[
    \|B_k\|
    \le
    \left(1+\frac{1}{\omega_{\rm SR1}}\right)^m\|B_{k,0}\|
    +
    \left( \left(1 + \frac{1}{\omega_{\rm SR1}} \right)^m - 1\right) b_m,
  \]
  where \(b_m\) is defined in \Cref{lem:bounded-secant-slopes-L0-L1}, and hence \(B_k\) satisfies \Cref{asm:Bk-L0-L1}.

  \paragraph{(iv) Modified limited-memory SR1}
  Using the same argument as in \citet{aravkin-baraldi-orban-2021},
  \[
    \|B_k\|
    \le
    \|B_{k,0}\|
    +
    \frac{m}{\bar \omega_{\rm SR1}}.
  \]

  \paragraph{(v) Limited-memory PSB}
  By the bound derived in \citet[Equation~\(2.13\)]{powell-2010}, iterating over at most \(m\) stored pairs and applying \Cref{lem:bounded-secant-slopes-L0-L1} shows that
  \[
    \|B_k\|
    \le
    \|B_{k,0}\|
    +
    3 m b_m,
  \]
  where \(b_m\) is defined in \Cref{lem:bounded-secant-slopes-L0-L1}, and hence \(B_k\) satisfies \Cref{asm:Bk-L0-L1}.

  \paragraph{(vi) Modified limited-memory PSB}
  Let \(z_{i_j}\defeq y_{i_j}-B_{k,j}s_{i_j}\).
  Then,
  \[
    \|B_{k,j+1}-B_{k,j}\|
    \le
    \frac{2\|z_{i_j}\|}{\|s_{i_j}\|}
    +
    \frac{|z_{i_j}^T s_{i_j}|}{\|s_{i_j}\|^2}
    \le
    \frac{3\|z_{i_j}\|}{\|s_{i_j}\|}.
  \]
  Iterating over at most \(m\) stored pairs and using the safeguard implies,
  \[
    \|B_k\|
    \le
    \|B_{k,0}\|
    +
    \frac{3m}{\bar\omega_{\rm PSB}}.
  \]
\end{proofEE}


\Cref{lem:limited-memory-qn-bounds-2} shows that the considered limited-memory quasi-Newton updates preserve \Cref{asm:Bk-L0-L1} whenever the initial matrix \(B_{k,0}\) satisfies this assumption.
Possible choices include \(B_{k,0}=\lambda I\), or, when \(\alpha\) is known,
\(B_{k,0}=\|\nabla f(x_k)\|^\alpha I\).

\subsection{Preliminary results}

Without loss of generality, we assume that \(L_0 \ge b_0\) and \(L_1 \ge b_1\), thereby avoiding the need to repeatedly take the maximum of the corresponding constants, where \(L_0,L_1,b_0\), and \(b_1\) are defined in \Cref{asm:smooth,asm:Bk-L0-L1}.

We will repeatedly use
\begin{equation}
  \label{eq:psi}
  \psialpha(t) \defeq \frac{t}{\tl{0} + \tl{1} \, t^\alpha}.
\end{equation}
Note that \(\psialpha\) is nondecreasing on \((0,\infty)\).
The following lemma states a consequence of the Cauchy decrease~\eqref{eq:suff-decrease} in terms of \(\psialpha\).

\begin{lemma}%
  \label{lem:sufficient-cauchy-decrease}
  Let \Cref{asm:problem-obj,asm:Bk-L0-L1} hold.
  For all \(k \in \N\),
  \begin{equation}%
    \label{eq:cauchy-sharp}
    m_{k}(0)-m_{k}\left(s_{k}\right) \geq  \kdc \|\nabla f(x_k)\| \min \left\{\psialpha(\|\nabla f(x_k)\|)\, , \, \Delta_k \right\},
  \end{equation}
  where \(\psialpha\) is defined in~\eqref{eq:psi}.
\end{lemma}


The following model-error estimate is a consequence of \Cref{lem:decrease-L0-L1}~\eqref{eq:function-difference} and \Cref{prop:galpha-properties}~\eqref{item:galpha-lower-bound}.

\begin{lemma}%
  \label{lem:model-error}
  Let \Cref{asm:problem-obj,asm:smooth,asm:Bk-L0-L1} hold.
  For all \(k \in \N\),
  \begin{equation}%
    \label{eq:model-error}
    |f(x_k + s_k) - m_k(s_k)| \le (\tl{0} + \tl{1} \|\nabla f(x_k)\|^\alpha) \, \galpha(\|s_k\|, \|\nabla f(x_k)\|) \|s_k\|^2,
  \end{equation}
  where \(\galpha\) is defined in~\Cref{lem:decrease-L0-L1}.
\end{lemma}


\section{Complexity analysis}%
\label{sec:unbounded-radius}

The complexity analysis of \Cref{alg:TR-GEN} relies on the following lemma, which establishes the well-definedness and monotonicity of the function \(\halpha\) that will be used to define a lower bound on \(\Delta_k\).

\begin{lemmaE}[Definition and monotonicity of \(\halpha\)][end]%
  \label{lem:h-alpha-properties}
  Let \(\alpha\in[0,1]\), and consider the function \(\galpha\) defined in \Cref{lem:decrease-L0-L1}.
  Define
  \begin{equation}
    \label{eq:def-Psi-alpha}
    \Psialpha(a,b)
    \defeq
    a \galpha(a,b),
    \qquad a,b>0.
  \end{equation}
  Then, for every \(b>0\), the mapping \(\Psialpha(\cdot,b)\) is a bijection from \((0,\infty)\) onto \((0,\infty)\).
  Consequently, \(\halpha:(0,\infty)^2 \to (0,\infty)\)
  \begin{equation}
    \label{eq:def-h-alpha}
    \halpha(b,t)
    \defeq
    \bigl(\Psialpha(\cdot,b)\bigr)^{-1}(t),
  \end{equation}
  is well-defined.
  Equivalently, \(\halpha(b,t)\) is the unique positive number satisfying
  \begin{equation}
    \label{eq:h-alpha-implicit-identity}
    \Psialpha\bigl(\halpha(b,t),b\bigr)=t.
  \end{equation}

  Moreover, \(\halpha\in\mathcal{C}^1((0,\infty)^2)\), is strictly increasing with respect to \(t\), and is nondecreasing with respect to \(b\).
  More precisely,
  \begin{equation}
    \label{eq:dt-h-alpha}
    \partial_t \halpha(b,t) =\frac{1}{\partial_a\Psialpha\bigl(\halpha(b,t),b\bigr)} > 0 \qquad \text{and}\qquad \partial_b \halpha(b,t) = - \frac{\partial_b\Psialpha\bigl(\halpha(b,t),b\bigr)}{\partial_a\Psialpha\bigl(\halpha(b,t),b\bigr)} \ge 0.
  \end{equation}
\end{lemmaE}

\begin{proofEE}
  Fix \(b>0\).
  By \Cref{prop:galpha-properties}~\eqref{item:galpha-monotone-a}, the mapping \(a\mapsto \galpha(a,b)\) is nondecreasing and satisfies \(\galpha(a,b)\ge 1\).
  Hence, for \(0<a_1<a_2\), \(\Psialpha(a_2,b)=a_2\galpha(a_2,b) \ge a_2\galpha(a_1,b) > a_1\galpha(a_1,b)=\Psialpha(a_1,b)\).
  Thus, \(\Psialpha(\cdot,b)\) is strictly increasing.

  Moreover, \(\lim_{a\downarrow0}\Psialpha(a,b) =\lim_{a\downarrow0}a\,\galpha(a,b)=0\).
  For \(\alpha>0\) this follows from \Cref{prop:galpha-properties}~\eqref{item:galpha-limit-a0}; for \(\alpha=0\), the convention \(0^0=1\) gives \(\galpha(a,b)=1+L_1/(L_0+L_1)\), which is constant in \(a\), so the limit is again \(0\).
  On the other hand, since \(\Psialpha(a,b)=a \galpha(a,b)\ge a\), we have \(\lim_{a\uparrow\infty}\Psialpha(a,b)=+\infty\).
  Since \(\Psialpha(\cdot,b)\) is continuous, it is a bijection from \((0,\infty)\) onto \((0,\infty)\).
  Consequently, \(\halpha\) is well-defined by~\eqref{eq:def-h-alpha}.

  According to \Cref{prop:galpha-properties}~\eqref{item:galpha-C1}, \(\Psialpha \in \mathcal{C}^1((0,\infty)^2)\).
  Furthermore, the monotonicity of \(\galpha(\cdot,b)\) in \Cref{prop:galpha-properties}~\eqref{item:galpha-monotone-a} implies \(\partial_a \galpha(a,b)\ge0\), and hence
  \begin{equation}
    \label{eq:partial-a-Psi-positive}
    \partial_a\Psialpha(a,b) = \galpha(a,b)+a\,\partial_a \galpha(a,b) \ge \galpha(a,b) \ge 1.
  \end{equation}

  Define \(F(a,b,t)\defeq\Psialpha(a,b)-t\).
  For every \(b, t > 0\), we have \(F(\halpha(b,t),b,t)=0\), while \(\partial_a F(\halpha(b,t),b,t)>0\) by~\eqref{eq:partial-a-Psi-positive}.
  The implicit function theorem therefore yields \(\halpha\in\mathcal{C}^1((0,\infty)^2)\).

  Differentiating~\eqref{eq:h-alpha-implicit-identity} with respect to \(t\) gives 
  \[\partial_a\Psialpha(\halpha(b,t),b)\, \partial_t \halpha(b,t)=1,\]
  which proves the first identity in~\eqref{eq:dt-h-alpha} because \(\partial_a\Psialpha(\halpha(b,t),b)>0\) by~\eqref{eq:partial-a-Psi-positive}.

  Similarly, differentiation with respect to \(b\) gives 
  \[\partial_a\Psialpha(\halpha(b,t),b)\, \partial_b \halpha(b,t)+\partial_b\Psialpha(\halpha(b,t),b)=0.\]
  By \Cref{prop:galpha-properties}~\eqref{item:galpha-monotone-b}, the mapping \(b\mapsto \galpha(a,b)\) is nonincreasing, so \(\partial_b\Psialpha(a,b) = a\,\partial_b \galpha(a,b)\le0\).
  This proves the second identity in~\eqref{eq:dt-h-alpha}.
\end{proofEE}

Now we derive a lower bound on \(\halpha(b,t)\).

\begin{lemmaE}[][end]%
  \label{lem:hb-bound-ag}
  For fixed \(b > 0\), let \(\halpha(b,\cdot)\) be defined as in~\eqref{eq:def-h-alpha}.
  Then, for all \(t>0\),
  \[
    \frac{t}{
      1+\dfrac{L_1 t^\alpha}{(L_0+L_1 b^\alpha)^{1-\alpha}}
    }
    \le \halpha(b,t) \le t .
  \]
\end{lemmaE}

\begin{proofEE}
  Fix \(b > 0\) and \(t>0\).
  Since \(\galpha(\halpha(b,t),b) \ge 1\) according to \Cref{prop:galpha-properties} and \(t = \halpha(b,t) \galpha(\halpha(b,t),b)\), we immediately obtain \(\halpha(b,t) \le t\).
  For brevity, write \(\galpha\defeq \galpha(\halpha(b,t),b)\).

  Assume first that \(0\le\alpha<1\).
  Let \(\mu \defeq L_0+L_1b^\alpha, \nu \defeq L_1/\mu^{1-\alpha}\).
  Then \(\galpha^{1-\alpha}=1+(1-\alpha)\nu \halpha(b,t)^\alpha\), so
  \[
    1+\nu t^\alpha = 1+\nu \halpha(b,t)^\alpha \galpha^\alpha = 1+\frac{\galpha-\galpha^\alpha}{1-\alpha}.
  \]
  By the concavity of \(x\mapsto x^\alpha\), \(\galpha^\alpha \le \alpha \galpha+1-\alpha\), hence \(\galpha \le 1 + \nu t^\alpha\).
  Since \(t=\halpha(b,t) \galpha\), it follows that
  \[
    \halpha(b,t) \ge \frac{t}{1+\nu t^\alpha}.
  \]

  For \(\alpha=1\), \(\galpha =e^{L_1h}\) and \(t=he^{L_1h}\).
  Using \(e^x\le1+xe^x\) for \(x\ge0\) gives \(e^{L_1h}\le1+L_1t\), and therefore
  \[
    h=\frac{t}{e^{L_1h}}
    \ge
    \frac{t}{1+L_1t}.
  \]
\end{proofEE}

For every \(k\) such that \(\nabla f(x_k)\neq0\), define
\begin{equation*}
  \theta_k \defeq
  \frac{\Delta_k}{
    \halpha\left(
    \min_{j\le k}\|\nabla f(x_j)\|,
    \kdc(1-\eta_2)
    \psialpha\left(\min_{j\le k}\|\nabla f(x_j)\|\right)
    \right)
  }.
\end{equation*}
If \(\nabla f(x_k)=0\) for some \(k\), then \(x_k\) is already a first-order stationary point.
Otherwise, both arguments of \(\halpha\) are positive, so \(\theta_k\) is well defined by \Cref{lem:h-alpha-properties}.
Next, we establish that \(\{\theta_k\}_{k \in \N}\) is uniformly bounded below.

\begin{lemma}%
  \label{lem:gamma-min}
  Let \Cref{asm:problem-obj,asm:smooth,asm:Bk-L0-L1} hold.
  For all \(k \in \N\),
  \begin{equation*}
    \theta_k \ge \theta_{\min} \defeq \gamma_1 \min\left\{ 1, \, \frac{\Delta_0}{\halpha(\|\nabla f(x_0)\|, \kdc (1 - \eta_2)\psialpha(\|\nabla f(x_0)\|))} \right\} > 0.
  \end{equation*}
  Moreover, if \(\tl{0} \ge 1\) or \(\tl{1} \ge 1\), then by choosing \(\Delta_0 \ge \max\{1, \|\nabla f(x_0)\|\}\), we have \(\theta_{\min} = \gamma_1\).
\end{lemma}

\begin{proof}
  First, observe that for \(k = 0\),
  \begin{align*}
    \theta_0 = \frac{\Delta_0 }{\halpha(\|\nabla f(x_0)\|, \kdc (1 - \eta_2)\psialpha(\|\nabla f(x_0)\|))} &
    \ge \theta_{\min},
  \end{align*}
  because \(\gamma_1 \leq 1\).

  Suppose, for contradiction, that \(\theta_{k+1}<\theta_{\min}\) for some index, and let \(k\) be the smallest such index; by the case \(k=0\) above and minimality, \(\theta_k\ge\theta_{\min}\).
  By Line~\ref{alg:TR-GEN:step-update}, \(\gamma_1 \Delta_{k} \le \Delta_{k+1}\), which implies that,
  \begin{align*}
    \frac{\Delta_{k}}{\halpha(\min_{j\le k+1} \|\nabla f(x_j)\|, \kdc (1 - \eta_2)\psialpha(\min_{j\le k+1}\|\nabla f(x_j)\|))} & \le \frac{\theta_{k+1}}{\gamma_1} < \frac{ \theta_{\min}}{\gamma_1} \le 1.
  \end{align*}
  Applying the increasing mapping \(\Psialpha(\cdot,\min_{j\le k+1}\|\nabla f(x_j)\|)\) to both sides and using the definition of \(\halpha\), we obtain
  \begin{align}
    \Psialpha(\Delta_k, \min_{j\le k+1} \|\nabla f(x_j)\|) & < \kdc (1 - \eta_2) \psialpha(\min_{j\le k+1}\|\nabla f(x_j)\|) \nonumber
    \\
    \label{eq:contradiction}
                                                               & \underset{(a)}{\le}  \kdc (1 - \eta_2)\psialpha(\|\nabla f(x_k)\|) \nonumber
    \\
                                                               & \underset{(b)}{\le} \psialpha(\|\nabla f(x_k)\|),
  \end{align}
  where (a) follows from the fact that \(\psialpha\) is increasing and (b) follows from the fact that \(\kdc (1 - \eta_2) \leq 1\).

  On the other hand, it follows from the definition of \(\rho_k\),
  \begin{align*}
    |\rho_k -1| & = \frac{|f(x_k + s_k) - m_k(s_k)|}{m_k(0) - m_k(s_k)}                                                                                                                                                            \\
                & \underset{(a)}{\le} \frac{ (\tl{0} + \tl{1} \|\nabla f(x_k)\|^\alpha) \galpha(\|s_k\|, \|\nabla f(x_k)\|) \|s_k\|^2}{\kdc \|\nabla f(x_k)\| \min \left\{\psialpha(\|\nabla f(x_k)\|), \, \Delta_k \right\}} \\
                & \underset{(b)}{=} \frac{ (\tl{0} + \tl{1} \|\nabla f(x_k)\|^\alpha) \Psialpha(\|s_k\|, \|\nabla f(x_k)\|) \|s_k\|}{\kdc \|\nabla f(x_k)\| \Delta_k}                                                          \\
                & \underset{(c)}{\le} \frac{ \Psialpha(\Delta_k, \min_{j\le k+1} \|\nabla f(x_j)\|)}{\kdc \psialpha(\|\nabla f(x_k)\|)} \underset{(d)}{\le} 1 - \eta_2,
  \end{align*}
  Here, (a) follows from \Cref{lem:model-error} together with the sharp form~\eqref{eq:cauchy-sharp} of the Cauchy decrease.
  Step (b) uses the definition \(\Psialpha(a,b)=a\,\galpha(a,b)\) in the numerator, and the identity \(\min\{\psialpha(\|\nabla f(x_k)\|),\Delta_k\}=\Delta_k\) in the denominator.
  The latter holds because \(\galpha\ge1\) by \Cref{prop:galpha-properties} (5), so that, writing \(b_{k+1}\defeq\min_{j\le k+1}\|\nabla f(x_j)\|\),
  \[
    \Delta_k
    \le
    \Delta_k\,\galpha(\Delta_k,b_{k+1})
    =
    \Psialpha(\Delta_k,b_{k+1})
    \underset{\eqref{eq:contradiction}}{<}
    \psialpha(\|\nabla f(x_k)\|).
  \]
  Step (c) follows from \(\|s_k\| \le \Delta_k\) together with the fact that \(\Psialpha\) is increasing with respect to its first argument and nonincreasing with respect to its second argument, and from the identity \((\tl{0}+\tl{1}\|\nabla f(x_k)\|^\alpha)/\|\nabla f(x_k)\| = 1/\psialpha(\|\nabla f(x_k)\|)\).
  Finally, (d) follows from inequality~\eqref{eq:contradiction}.
  Therefore, \(\rho_k \ge \eta_2\), which implies that iteration \(k\) is very successful, and \(\Delta_{k+1} \ge \Delta_k\).
  Thus, by the monotonicity of \(\halpha\) established in \Cref{lem:h-alpha-properties} and of \(\psialpha\), we have
  \begin{align*}
     & \halpha(\min_{j\le k+1} \|\nabla f(x_j)\|, \kdc (1 - \eta_2)\psialpha(\min_{j\le k+1}\|\nabla f(x_j)\|))  \\
     & \le \halpha(\min_{j\le k} \|\nabla f(x_j)\|, \kdc (1 - \eta_2)\psialpha(\min_{j\le k}\|\nabla f(x_j)\|)).
  \end{align*}
  Hence, \(\theta_{k} \le \theta_{k+1} < \theta_{\min}\), which contradicts the assumption that \(\theta_k \ge \theta_{\min}\).

  For the last statement, by \Cref{lem:hb-bound-ag},
  \begin{align*}
    \halpha\!\left(\|\nabla f(x_0)\|,\kdc(1-\eta_2)\psialpha(\|\nabla f(x_0)\|)\right)
     & \le
    \kdc(1-\eta_2)\psialpha(\|\nabla f(x_0)\|) \\
     & \le
    \psialpha(\|\nabla f(x_0)\|)               \\
     & \le
    \max\{1,\|\nabla f(x_0)\|\},
  \end{align*}
  Then, if \(\tl{0}\ge1\), then \(\psialpha(\|\nabla f(x_0)\|)\le\|\nabla f(x_0)\|/\tl{0}\le\|\nabla f(x_0)\|\).
  If instead \(\tl{1}\ge1\), then \(\psialpha(\|\nabla f(x_0)\|)\le\|\nabla f(x_0)\|^{1-\alpha}/\tl{1}\le\|\nabla f(x_0)\|^{1-\alpha}\le\max\{1,\|\nabla f(x_0)\|\}\), since \(1-\alpha\in[0,1]\).

  Hence, if \(\Delta_0\ge\max\{1,\|\nabla f(x_0)\|\}\), then \(\theta_{\min}=\gamma_1\). 
\end{proof}

In general, \(\theta_{\min}\) may depend on the initial gradient norm \(\|\nabla f(x_0)\|\).
This dependence can be avoided by choosing \(\Delta_0\) sufficiently large, as indicated in \Cref{lem:gamma-min}.

The following lemma provides a model-decrease estimate.

\begin{lemma}%
  \label{lem:sufficient-decrease-general}
  Let \Cref{asm:problem-obj,asm:smooth,asm:Bk-L0-L1} hold.
  For all \(k \in \N\),
  \begin{align*}
     & m_{k}(0)-m_{k}\left(s_{k}\right) \ge                                                                                                                      \\
     & \quad \kdc \|\nabla f(x_k)\|  \halpha(\min_{j\le k} \|\nabla f(x_j)\|, \kdc (1 - \eta_2)\psialpha(\min_{j \le k}\|\nabla f(x_j)\|)) \theta_{\min}.
  \end{align*}
\end{lemma}

\begin{proof}
  \Cref{lem:sufficient-cauchy-decrease} gives
  \begin{align*}
     & m_{k}(0)-m_{k}\left(s_{k}\right)  \ge  \kdc \|\nabla f(x_k)\| \min \left\{\psialpha(\|\nabla f(x_k)\|)\, , \, \Delta_k \right\}                                                                   \\
     & \underset{(a)}{\ge}  \kdc \|\nabla f(x_k)\| \min \left\{\kdc (1 - \eta_2)\psialpha(\min_{j\le k}\|\nabla f(x_j)\|)\, , \, \Delta_k \right\}                                                       \\
     & \underset{(b)}{\ge}  \kdc \|\nabla f(x_k)\| \min \left\{\halpha(\min_{j\le k} \|\nabla f(x_j)\|, \kdc (1 - \eta_2)\psialpha(\min_{j \le k}\|\nabla f(x_j)\|)) \theta_{\min}, \Delta_k \right\} \\
     & \underset{(c)}{=}  \kdc \|\nabla f(x_k)\| \halpha(\min_{j\le k} \|\nabla f(x_j)\|,\kdc (1 - \eta_2)\psialpha(\min_{j\le k}\|\nabla f(x_j)\|)) \theta_{\min},
  \end{align*}
  where (a) follows from the fact that \(\kdc (1 - \eta_2) \leq 1\) and that \(\psialpha\) is nondecreasing, (b) follows from \(\halpha(b,t)\le t\) and \(\theta_{\min}\le1\), and (c) follows from \Cref{lem:gamma-min}.
\end{proof}

The lower bound in \Cref{lem:sufficient-decrease-general} is key to the complexity analysis in \Cref{sec:complexity-general}.

If \Cref{alg:TR-GEN} has finitely many successful iterations, it identifies a first-order critical point in finite time.

\begin{theorem}
  Let \Cref{asm:problem-obj,asm:smooth,asm:Bk-L0-L1} hold.
  If \Cref{alg:TR-GEN} generates finitely many successful iterations, then \(x_k = x^*\) for all sufficiently large \(k\), and \(\nabla f(x^*) = 0\).
\end{theorem}

\begin{proof}
  Let \(k_f\) be the last successful iteration. Then \(x_k = x_{k_f +1}\) for all \(k \geq k_f + 1\). Set \(x^* \defeq x_{k_f + 1}\).
  Suppose, by contradiction, that \(\|\nabla f(x_k)\| \geq \nu > 0\) for all \(k \in \N\).
  By \Cref{lem:gamma-min} and the monotonicity of \(\halpha\) and \(\psialpha\), 
  \begin{equation*}
    \Delta_{k}
    \ge
    \theta_{\min} \halpha\bigl(\nu, \kdc(1-\eta_2)\psialpha(\nu)\bigr).
  \end{equation*}
  Thus, \(\Delta_k\) is bounded away from zero for all \(k \geq k_f + 1\). 
  On the other hand, all iterations after \(k_f + 1\) are unsuccessful, so the update mechanism of \Cref{alg:TR-GEN} yields \(\Delta_k \to 0\), a contradiction.
  Hence, \(\liminf_{k \to \infty} \|\nabla f(x_k)\| = \|\nabla f(x^*)\| = 0\).
\end{proof}

\subsection{Complexity on general objectives}%
\label{sec:complexity-general}

Let \(0 < \epsilon < 1\) and \(k_{\epsilon}\) be the first iteration of \Cref{alg:TR-GEN} with \(\|\nabla f(x_{k_{\epsilon}})\| > \epsilon\), and \(\|\nabla f(x_{k_{\epsilon} + 1})\| \le \epsilon\).
Define
\begin{align*}
  \mathcal{S}(\epsilon) & \defeq \mathcal{S}_{k_{\epsilon}-1} = \{ k \in \mathcal{S} \mid k < k_{\epsilon} \}.
\end{align*}

\begin{theorem}%
  \label{thm:complexity:S}
  Let \Cref{asm:problem-obj,asm:smooth,asm:Bk-L0-L1} hold.
  Assume that \Cref{alg:TR-GEN} generates infinitely many successful iterations.
  Let  \(\theta_{\min}\) be as in \Cref{lem:gamma-min}.
  Define
  \begin{align}%
    \label{eq:kappa1-kappa2-Delta-general}
    \kappa_{0} & \defeq  \frac{f(x_0) - f_{\textup{low}}}{\eta_1 \kdc^2 (1 - \eta_2) \theta_{\min}}.
  \end{align}
  Then, for all \(\alpha \in [0,1]\),
  \begin{equation}
    \label{eq:S-eps0}
    |\mathcal{S} (\epsilon)| \leq \tl{1} \left(1 + (\kdc (1 - \eta_2))^\alpha\right) \kappa_0 \; \epsilon^{\alpha-2} +  \tl{0} \, \kappa_0 \; \epsilon^{-2}.
  \end{equation}
\end{theorem}

\begin{proof}
  If \(\mathcal{S}(\epsilon) = \varnothing \), then the theorem holds trivially as \(|\mathcal{S}(\epsilon)| = 0\).

  Otherwise, let \(\ell \in \mathcal{S}(\epsilon)\).
  By \Cref{lem:sufficient-decrease-general}, we have
  \begin{align*}
     & f(x_{\ell}) - f(x_{\ell} + s_{\ell}) \geq \eta_1 \left(m_{\ell}(0) - m_{\ell}(s_{\ell})\right)                                                                                \\
     & \quad \geq \eta_1 \kdc \|\nabla f(x_\ell)\|  \halpha(\min_{j\le \ell} \|\nabla f(x_j)\|, \kdc (1 - \eta_2)\psialpha(\min_{j \le \ell}\|\nabla f(x_j)\|)) \theta_{\min} \\
     & \quad \underset{(a)}{\geq} \eta_1 \kdc \epsilon  \halpha(\epsilon, \kdc (1 - \eta_2)\psialpha(\epsilon)) \theta_{\min},
  \end{align*}
  where (a) follows from the fact that \(\psialpha\) is increasing and so is \(\halpha\) with respect to both of its arguments, and that \(\|\nabla f(x_j)\| \geq \epsilon\) for all \(j = 0, \ldots, \ell\).
  Using \Cref{lem:hb-bound-ag},
  \begin{align*}
    f(x_{\ell}) - f(x_{\ell} + s_{\ell}) & \geq \eta_1 \kdc \epsilon  \frac{\dfrac{\kdc (1 - \eta_2)\epsilon}{\tl{0} + \tl{1} \epsilon^\alpha}}{1 + \dfrac{\tl{1} (\kdc (1 - \eta_2) \epsilon/(\tl{0} + \tl{1} \epsilon^\alpha))^\alpha}{(\tl{0} + \tl{1} \epsilon^\alpha)^{\,1-\alpha}}} \theta_{\min} \\
                                         & \geq  \eta_1 \kdc^2 \epsilon^2 \theta_{\min} (1 - \eta_2) \frac{1}{\tl{0} + \tl{1} \epsilon^\alpha + \tl{1} (\kdc (1 - \eta_2) \epsilon)^\alpha},
  \end{align*}

  We now sum the above inequality over \(\ell \in \mathcal{S}(\epsilon)\).
  The sequence \(\{f(x_k)\}\) is nonincreasing over all iterations, since \(x_{k+1}=x_k\) at an unsuccessful iteration and \(f(x_{k+1})<f(x_k)\) at a successful one.
  Hence the decreases \(f(x_\ell)-f(x_{\ell+1})\), \(\ell\in\mathcal{S}(\epsilon)\), telescope into at most \(f(x_0)-f_{\textup{low}}\), and we obtain
  \begin{align*}
    f(x_0) - f_{\textup{low}} & \geq \frac{ \eta_1 \kdc^2 \epsilon^2 \theta_{\min} (1 - \eta_2) |\mathcal{S}(\epsilon)|}{ \tl{0} + \tl{1} (1 + (\kdc (1 - \eta_2))^\alpha) \epsilon^\alpha},
  \end{align*}
  which gives the desired result.
\end{proof}

Several remarks on \Cref{thm:complexity:S} are in order.
Firstly, the constants and the bound are well-defined for every \(\alpha \in [0,1]\), where \(\alpha\) is the parameter in \Cref{asm:smooth}.
Secondly, the bound in \Cref{thm:complexity:S} recovers the best known gradient-descent complexity bound without requiring any estimate of the problem parameters; see \Cref{tab:L0L1-complexity}.
Thirdly, the bound in \Cref{thm:complexity:S} is of order \(O(\epsilon^{-2})\) when \(\alpha = 0\) and of order \(O(\tl{0}\epsilon^{-2} + \tl{1}\epsilon^{-1})\) when \(\alpha = 1\), which matches the best known dependence under \((1,L_0,L_1)\)-smoothness.

\Cref{thm:complexity:S} bounds the number of \emph{successful} iterations.
The total number of iterations, including unsuccessful ones, is bounded in the following theorem.

\begin{theorem}%
  \label{thm:total-iterations}
  Let the assumptions of \Cref{thm:complexity:S} hold, assume
  \(k_\epsilon \ge 1\), and let
  \[
    c_\alpha
    \defeq
    \frac{a}{
      1+\dfrac{L_1 a^\alpha}
      {(L_0+L_1)^{1-\alpha}}
    },
    \qquad
    a \defeq \frac{\kdc(1-\eta_2)}{\tl{0}+\tl{1}},
    \qquad
    \kappa_0
    \defeq
    \frac{f(x_0)-f_{\mathrm{low}}}
    {\eta_1\kdc^2(1-\eta_2)\theta_{\min}}.
  \]
  Then,
  \begin{align*}
    k_\epsilon
    \le{}&
    \bigl(1+|\log_{\gamma_2}(\gamma_3)|\bigr)
    \kappa_0
    \left[
      \tl{1}
      \left(1+(\kdc(1-\eta_2))^\alpha\right)
      \epsilon^{\alpha-2}
      +
      \tl{0}\epsilon^{-2}
    \right]
    \\
    &+
      \log_{\gamma_2}\left(
        \frac{\theta_{\min}c_\alpha\epsilon}{\Delta_0}
      \right).
  \end{align*}
\end{theorem}

\begin{proof}
  Since the iterations \(0,\ldots,k_\epsilon-1\) are either successful
  or unsuccessful,
  \(
    k_\epsilon
    =
    |\mathcal{S}(\epsilon)|
    +
    |\mathcal{U}_{k_\epsilon-1}|.
  \)
  By \Cref{lem:bound-unsuccessful-iterations},
  \[
    k_\epsilon
    \le
    \bigl(1+|\log_{\gamma_2}(\gamma_3)|\bigr)
    |\mathcal{S}(\epsilon)|
    +
      \log_{\gamma_2}\left(
        \frac{\Delta_{k_\epsilon}}{\Delta_0}
      \right).
  \]

  Hence, by \Cref{lem:gamma-min}, the fact that \(\min_{j\le k_\epsilon}\|\nabla f(x_j)\|\ge\epsilon\) and the monotonicity of \(\halpha\) and \(\psialpha\),
  \[
    \Delta_{k_\epsilon}
    \ge
    \theta_{\min}
    \halpha\bigl(
      \epsilon,
      \kdc(1-\eta_2)\psialpha(\epsilon)
    \bigr).
  \]
  Moreover, because \(\epsilon<1\), we have
  \[
    \psialpha(\epsilon)
    \ge
    \frac{\epsilon}{\tl{0}+\tl{1}},
  \]
  and therefore, by \Cref{lem:hb-bound-ag},
  \[
    \halpha\bigl(
      \epsilon,
      \kdc(1-\eta_2)\psialpha(\epsilon)
    \bigr)
    \ge
    c_\alpha\epsilon.
  \]
  The result now follows by the decreasing property of \(\log_{\gamma_2}\), since \(\gamma_2\in(0,1)\), and \Cref{thm:complexity:S}.
\end{proof}



\subsection{Complexity for convex and strongly convex objectives}%
\label{sec:convex-S}

In this section, we derive complexity results of \Cref{alg:TR-GEN} for both convex and strongly convex functions under the \((\alpha ,L_0, L_1)\)-smoothness assumption.

Convexity leads to improved complexity bounds.
We focus on successful iterations, since  \Cref{lem:bound-unsuccessful-iterations} then provides a bound on the unsuccessful iterations.

For the convex case, we consider the optimality measure

\begin{equation}%
  \label{eq:Dk}
  \delta_k \defeq f(x_k) - f_{\rm low},
\end{equation}
where \(f_{\rm low}\) is the infimum of \Cref{asm:problem-obj}.

Let \(0 < \epsilon < 1\) and, since \(\{\delta_k\}_{k \in \N}\) is nonincreasing, let \(\widehat{k}_{\epsilon}\) be the first iteration of \Cref{alg:TR-GEN} with \(\delta_{\widehat{k}_{\epsilon}} > \epsilon\), and \(\delta_{\widehat{k}_{\epsilon}+1}\le\epsilon\).
Define
\begin{align*}
  \widehat{\mathcal{S}}(\epsilon) & \defeq \mathcal{S}_{\widehat{k}_{\epsilon}-1} = \{ k \in \mathcal{S} \mid k < \widehat{k}_{\epsilon} \}.
\end{align*}

We impose the following condition, which is standard in the analysis of convex optimization algorithms, and is also used in \citep{cartis-gould-toint-2012}.

\begin{problemassumption}%
  \label{asm:convex}
  The function \(f\) is convex, and there exists \(R > 0\) such that \(\|x - x^*\| \le R\), for every \(x\) satisfying \(f(x) \le f(x_0)\), where \(x^*\) is a minimizer of \(f\).
\end{problemassumption}

The following property from \citep{cartis-gould-toint-2012} relates the objective gap to the gradient norm.

\begin{lemma}[{\citealp[Lemma~\(2.4\)]{cartis-gould-toint-2012}}]%
  \label{lem:convex}
  Assume that \Cref{asm:convex} holds.
  Then, for all \(k \in \N \), \(\delta_k \le R \| \nabla f(x_k) \|\).
\end{lemma}

We now state the main complexity result for convex objectives satisfying \((\alpha,L_0,L_1)\)-smoothness.

\begin{theorem}%
  \label{thm:complexity:S-convex-2}
  Let \Cref{asm:problem-obj,asm:smooth,asm:convex,asm:Bk-L0-L1} hold.
  Assume that \Cref{alg:TR-GEN} generates infinitely many successful iterations.
  Let \(\theta_{\min}\) be as in \Cref{lem:gamma-min}.
  Define
  \[
    \widehat{\kappa}_{0} \defeq \frac{R^2}{\eta_1\kdc^2(1-\eta_2)\theta_{\min}},\quad
    \widehat{\kappa}_{\alpha} \defeq \frac{\tl{1}\bigl(1+(\kdc(1-\eta_2))^\alpha\bigr)}{R^{\alpha}\,(1-\alpha)}\,\widehat{\kappa}_0,
  \]
  \[
    \widehat{\kappa}_{1} \defeq \frac{\tl{1}\bigl(1+\kdc(1-\eta_2)\bigr)}{R}\,\widehat{\kappa}_0,\quad
    \widehat{\kappa}_{2} \defeq \tl{0}\widehat{\kappa}_0.
  \]
  Then, if \(\alpha \in [0,1)\),
  \begin{equation}%
    \label{eq:S-eps0-convex-unbounded}
    |\widehat{\mathcal{S}} (\epsilon)| \leq \widehat{\kappa}_\alpha \; \epsilon^{\alpha-1} +  \widehat{\kappa}_2 \; \epsilon^{-1} -
    \widehat{\kappa}_\alpha \; \delta_0^{\alpha-1}.
  \end{equation}
  If \(\alpha = 1\),
  \begin{equation}%
    \label{eq:S-eps0-convex-alpha-1-unbounded}
    |\widehat{\mathcal{S}} (\epsilon)| \leq \widehat{\kappa}_1 \; \log(\delta_0/\epsilon) +  \widehat{\kappa}_2 \; \epsilon^{-1}.
  \end{equation}
\end{theorem}

\begin{proof}
  If \(\widehat{\mathcal{S}}(\epsilon) = \emptyset\), then the theorem holds trivially as \(|\widehat{\mathcal{S}}(\epsilon)| = 0\).
  Otherwise, consider some \(\ell \in \widehat{\mathcal{S}}(\epsilon)\).
  By \Cref{lem:sufficient-decrease-general}, we have
  \begin{align*}
     & f(x_{\ell}) - f(x_{\ell} + s_{\ell}) \geq \eta_1 \left(m_{\ell}(0) - m_{\ell}(s_{\ell})\right)                                                                                 \\
     & \quad \geq \eta_1 \kdc \|\nabla f(x_\ell)\|  \halpha(\min_{j\le \ell} \|\nabla f(x_j)\|, \kdc (1 - \eta_2)\psialpha(\min_{j \le \ell}\|\nabla f(x_j)\|)) \theta_{\min}.
  \end{align*}
  Using \Cref{lem:hb-bound-ag}, we have
  \begin{align*}
     & f(x_{\ell}) - f(x_{\ell} + s_{\ell})                                                                                                                                                                                                                                                                                                                                                                          \\
     & \geq   \frac{\dfrac{\eta_1 \kdc^2 \|\nabla f(x_\ell)\| (1 - \eta_2)\min_{j\le \ell} \|\nabla f(x_j)\|}{\tl{0} + \tl{1} \min_{j\le \ell} \|\nabla f(x_j)\|^\alpha}}{1 + \dfrac{\tl{1} (\kdc (1 - \eta_2) \min_{j\le \ell} \|\nabla f(x_j)\|/(\tl{0} + \tl{1} \min_{j\le \ell} \|\nabla f(x_j)\|^\alpha))^\alpha}{(\tl{0} + \tl{1} \min_{j\le \ell} \|\nabla f(x_j)\|^\alpha)^{\,1-\alpha}}} \theta_{\min} \\
     & \geq  \eta_1 \kdc^2 (1 - \eta_2)  \frac{\min_{j\le \ell} \|\nabla f(x_j)\|^2}{\tl{0} + \tl{1} (1 + (\kdc (1 - \eta_2))^\alpha) \min_{j\le \ell} \|\nabla f(x_j)\|^\alpha}  \theta_{\min}.
  \end{align*}
  Since \(\delta_j\) is nonincreasing, \Cref{lem:convex} implies that
  \begin{align*}
    \delta_j \le R \|\nabla f(x_j)\| & \implies \delta_\ell = \min_{j = 0, \ldots,\ell} \delta_j \le R \min_{j = 0, \ldots,\ell} \|\nabla f(x_j)\|.
  \end{align*}
  Thus, recalling~\eqref{eq:Dk}, we obtain
  \begin{equation*}
    \delta_\ell - \delta_{\ell + 1} = f(x_\ell) - f(x_\ell + s_\ell)  \geq  \frac{\eta_1 \kdc^2 (1 - \eta_2) (\delta_\ell/R)^2}{\tl{0} + \tl{1} (1 + (\kdc (1 - \eta_2))^\alpha) (\delta_\ell/R)^\alpha}  \theta_{\min}.
  \end{equation*}
  We set \(\omega \defeq \eta_1 \kdc^2 (1 - \eta_2)\theta_{\min} \) and \(\tl{\alpha} \defeq \tl{1} (1 + (\kdc (1 - \eta_2))^\alpha)\) and define the function \(\varphi: \R_+ \to \R_+\) as \( \varphi(t) \defeq \frac{t^2}{\tl{0} + \tl{\alpha} t^\alpha}\), so that \(\widehat\kappa_0 = R^2/\omega\).
  Therefore, we have \(\delta_\ell - \delta_{\ell+1}  \geq \omega \, \varphi(\delta_\ell/R)\).
  The function \(\varphi\) is increasing.
  Hence, for every \(t \in [\delta_{\ell+1}, \delta_\ell]\),
  \begin{align*}
    \omega \le \frac{\delta_\ell - \delta_{\ell+1}}{\varphi(\delta_\ell/R)} & \le \int_{\delta_{\ell+1}}^{\delta_\ell} \frac{1}{\varphi(t/R)} dt = \int_{\delta_{\ell+1}}^{\delta_{\ell}} \left(\frac{\tl{0} R^2}{t^2} + \frac{\tl{\alpha} R^{2-\alpha}}{t^{2-\alpha}}\right) dt.
  \end{align*}
  We sum this inequality over \(\ell \in \widehat{\mathcal S}(\epsilon)\).
  Since \(\{\delta_k\}\) is nonincreasing over all iterations, the intervals \([\delta_{\ell+1},\delta_\ell]\), \(\ell\in\widehat{\mathcal S}(\epsilon)\), are pairwise disjoint and contained in \([\delta_{\widehat{k}_{\epsilon}},\delta_0]\), so that
  \begin{align*}
    \omega |\widehat{\mathcal{S}}(\epsilon)| & \le \int_{\delta_{\widehat{k}_{\epsilon}}}^{\delta_0} \left(\frac{\tl{0} R^2}{t^2} + \frac{\tl{\alpha} R^{2-\alpha}}{t^{2-\alpha}}\right) dt.
  \end{align*}
  Hence, if \(\alpha \in [0,1)\), then using the fact that \(\delta_{\widehat{k}_{\epsilon}} > \epsilon\),
  \begin{align*}
    \omega |\widehat{\mathcal{S}}(\epsilon)| & \le \tl{0} R^2 \epsilon^{-1} + \frac{\tl{\alpha} R^{2-\alpha}}{1-\alpha} (\epsilon^{\alpha-1} - \delta_0^{\alpha-1})
  \end{align*}
  If \(\alpha = 1\), then
  \begin{align*}
    \omega |\widehat{\mathcal{S}}(\epsilon)| & \le \tl{0} R^2 \epsilon^{-1} + \tl{\alpha} R \log(\delta_0/\epsilon).
  \end{align*}
\end{proof}

We next specialize the analysis to strongly convex objectives.

\begin{problemassumption}%
  \label{asm:str-convex}
  There exists \(\mu > 0\) such that,
  \[
    f(y) \ge f(x) + \nabla f(x)^\top (y - x) + \frac{1}{2} \mu \|y - x\|^2 \quad \text{for all } x, y \in \R^n.
  \]
\end{problemassumption}

Strong convexity yields a sharper relation between the distance to the global minimum and the gradient norm, which yields a better complexity bound.

\begin{lemma}[{\citealp[Lemma~\(2.6\)]{cartis-gould-toint-2012}}]%
  \label{lem:str-convex}
  Assume that \Cref{asm:str-convex} holds.
  Then, for all \(k \in \N\), \(\delta_k \le \frac{1}{2\mu} \| \nabla f(x_k) \|^2\).
\end{lemma}

We now state the main complexity result for strongly convex objectives.

\begin{theorem}%
  \label{thm:complexity:S-str-convex-2}
  Let \Cref{asm:problem-obj,asm:smooth,asm:str-convex,asm:Bk-L0-L1} hold.
  Assume that \Cref{alg:TR-GEN} generates infinitely many successful iterations.
  Let \(\theta_{\min}\) be as in \Cref{lem:gamma-min}.
  Define
  \[
    \begin{aligned}
      \widehat{\kappa}_{0} & \defeq \frac{1}{2 \mu \eta_1 \kdc^2 (1-\eta_2)\theta_{\min}}, \qquad
      \widehat{\kappa}_{\alpha} \defeq \tl{1}\bigl(1+(\kdc(1-\eta_2))^\alpha\bigr)\widehat{\kappa}_0,
    \end{aligned}
  \]
  Then if \(\alpha = 0\),
  \begin{equation}%
    |\widehat{\mathcal{S}} (\epsilon)| \le (\tl{0} + 2 \tl{1}) \widehat{\kappa}_{0} \log(\delta_0/\epsilon).
  \end{equation}
  Otherwise, if \(\alpha \in (0,1]\),
  \begin{equation}%
    |\widehat{\mathcal{S}} (\epsilon)| \le \tl{0} \widehat{\kappa}_{0} \log(\delta_0/\epsilon) + \frac{2^{1+\alpha/2} \, \widehat{\kappa}_{\alpha} \, \mu^{\alpha/2}}{\alpha} \, (\delta_0^{\alpha/2} - \epsilon^{\alpha/2}).
  \end{equation}
\end{theorem}

\begin{proof}
  If \(\widehat{\mathcal{S}}(\epsilon) = \emptyset\), then the theorem holds trivially as \(|\widehat{\mathcal{S}}(\epsilon)| = 0\).
  Otherwise, consider some \(\ell \in \widehat{\mathcal{S}}(\epsilon)\).
  By \Cref{lem:sufficient-decrease-general,lem:hb-bound-ag}, we have
  \begin{align*}
     & f(x_{\ell}) - f(x_{\ell} + s_{\ell})  \ge    \frac{\eta_1 \kdc^2 (1 - \eta_2) \min_{j\le \ell} \|\nabla f(x_j)\|^2}{\tl{0} + \tl{1} (1 + (\kdc (1 - \eta_2))^\alpha) \min_{j\le \ell} \|\nabla f(x_j)\|^\alpha}  \theta_{\min}.
  \end{align*}
  Furthermore, due to the decreasing property of \(\delta_k\) and the application of \Cref{lem:str-convex}, we conclude that for all \(j \in \{0, \ldots, \ell\}\),
  \begin{align*}
    \delta_j \le \frac{1}{2\mu} \|\nabla f(x_j)\|^2 & \implies \delta_\ell = \min_{j = 0, \ldots,\ell} \delta_j \le \frac{1}{2\mu} \min_{j = 0, \ldots,\ell} \|\nabla f(x_j)\|^2 .
  \end{align*}
  Thus, recalling~\eqref{eq:Dk}, we obtain
  \begin{equation*}
    \delta_\ell - \delta_{\ell + 1} = f(x_\ell) - f(x_\ell + s_\ell) \geq  \frac{2 \eta_1 \kdc^2 (1 - \eta_2) \mu \delta_\ell}{\tl{0} + \tl{1} (1 + (\kdc (1 - \eta_2))^\alpha) \sqrt{2\mu \delta_\ell}^\alpha}  \theta_{\min}.
  \end{equation*}
  Note that the map \(m \mapsto m^2/(\tl{0}+\tl{\alpha}m^\alpha)\) is increasing on \((0,\infty)\), which is what licenses substituting the lower bound \(\min_{j\le\ell}\|\nabla f(x_j)\| \ge \sqrt{2\mu\delta_\ell}\) in the display above.
  We set \(\omega \defeq \eta_1 \kdc^2 (1 - \eta_2)\theta_{\min} \) and \(\tl{\alpha} \defeq \tl{1} (1 + (\kdc (1 - \eta_2))^\alpha)\) and define the function \(\varphi: \R_+ \to \R_+\) as \(\varphi(t) \defeq \frac{t}{\tl{0} + \tl{\alpha} (\sqrt{2\mu t})^\alpha}\), so that \(\widehat\kappa_0 = 1/(2\mu\omega)\).
  Therefore, we have \(\delta_\ell - \delta_{\ell+1} \geq 2 \mu \omega \, \varphi(\delta_\ell)\).
  The function \(\varphi\) is increasing.
  Hence, for every \(t \in [\delta_{\ell+1}, \delta_\ell]\),
  \begin{align*}
    2 \mu \omega \le \frac{\delta_\ell - \delta_{\ell+1}}{\varphi(\delta_\ell)} & \le \int_{\delta_{\ell+1}}^{\delta_\ell} \frac{1}{\varphi(t)} dt  = \int_{\delta_{\ell+1}}^{\delta_{\ell}} \left(\frac{\tl{0}}{t} + \frac{\tl{\alpha} (2\mu)^{\alpha/2}}{t^{1-\alpha/2}}\right) dt.
  \end{align*}
  We sum this inequality over \(\ell \in \widehat{\mathcal S}(\epsilon)\).
  As above, \(\{\delta_k\}\) is nonincreasing over all iterations, so the intervals \([\delta_{\ell+1},\delta_\ell]\), \(\ell\in\widehat{\mathcal S}(\epsilon)\), are pairwise disjoint and contained in \([\delta_{\widehat{k}_{\epsilon}},\delta_0]\), whence
  \begin{align*}
    2 \mu \omega |\widehat{\mathcal{S}}(\epsilon)| & \le \int_{\delta_{\widehat{k}_{\epsilon}}}^{\delta_0} \left(\frac{\tl{0}}{t} + \frac{\tl{\alpha} (2\mu)^{\alpha/2}}{t^{1-\alpha/2}}\right) dt.
  \end{align*}
  If \(\alpha=0\), using \(\delta_{\widehat k_\epsilon}>\epsilon\), we have
  \begin{align*}
    2 \mu \omega |\widehat{\mathcal{S}}(\epsilon)| & \le (\tl{0} + 2 \tl{1}) \log(\delta_0/\epsilon).
  \end{align*}
  Otherwise, if \(\alpha \in (0,1]\), then
  \begin{align*}
    2 \mu \omega |\widehat{\mathcal{S}}(\epsilon)| & \le \tl{0} \log(\delta_0/\epsilon) + \frac{2^{1+\alpha/2} \, \tl{\alpha} \, \mu^{\alpha/2}}{\alpha} \, (\delta_0^{\alpha/2} - \epsilon^{\alpha/2}),
  \end{align*}
  which gives the desired result.
\end{proof}

Several remarks on \Cref{thm:complexity:S-convex-2,thm:complexity:S-str-convex-2} are in order.
For \(\alpha=0\), which corresponds to the \(L\)-smooth case, where \(L = L_0 + L_1\), we recover the standard \(O(L \epsilon^{-1})\) convex and \(O(L \log(\epsilon^{-1}))\) strongly convex bounds.
For \(\alpha\in(0,1)\), our bounds are \(O(L_0\epsilon^{-1}+L_1\epsilon^{\alpha-1})\) and \(O(L_0\log(\epsilon^{-1})+L_1)\), respectively, which improves the dependence on \(L_1\) in \citep{li2023convex}; see \Cref{tab:L0L1-complexity-alpha}. 
The strongly convex bound appears to be new.
For \(\alpha=1\), the convex bound \(O(L_0\epsilon^{-1}+L_1\log(\epsilon^{-1}))\) matches the order obtained for gradient descent in \citep{vankov2024optimizing}, but without requiring prior knowledge of \(L_0\) or \(L_1\).
In the convex case, our constants in \Cref{tab:L0L1-complexity,tab:L0L1-complexity-alpha} differ slightly from those in \citep{tyurin2025unifiedtheorygradientdescent,li2023convex,vankov2024optimizing,gorbunov2025methods,koloskova2023revisiting} with respect to the radius \(R\): these works state their bounds in terms of \(R_0 = \|x_0 - x^\star\|\), whereas ours involve the radius \(R \ge R_0\) of the sublevel set \(\{x : f(x) \le f(x_0)\}\) around \(x^\star\) (\Cref{asm:convex}).
This is because, unlike the gradient methods analyzed in those works, trust-region methods do not guarantee that \(\|x_k - x^\star\|\) is nonincreasing; they only guarantee that \(f(x_k)\) is, so that the iterates remain in this sublevel set.
In the strongly convex case, our constants in \Cref{tab:L0L1-complexity,tab:L0L1-complexity-alpha} also differ from those in \citep{koloskova2023revisiting,gorbunov2025methods,li2023convex} because we measure convergence by the objective gap rather than the distance to a solution.
Despite this difference, the resulting complexity bounds have the same order in \(\epsilon\), while the displayed \(L_1\)-dependent term is linear in \(L_1\), whereas \citet{gorbunov2025methods}'s is quadratic.

\subsection{Sharpness of the complexity bound}%

Let $\alpha\in[0,1]$ and let $\epsilon\in(0,\frac12]$ denote the desired accuracy.
Define
\begin{equation}
  \label{eq:def-Neps-alpha}
  g_\epsilon \defeq 2\epsilon,
  \qquad
  \bar k_\epsilon
  \defeq
  \left\lfloor 2g_\epsilon^{\alpha-2}\right\rfloor+1 = \left\lfloor 2^{\alpha-1}\epsilon^{\alpha-2}\right\rfloor+1.
\end{equation}
Our goal is to construct a continuously differentiable, bounded-below, $(\alpha,0,1)$-smooth function $f_{\alpha,\epsilon}\colon\R\to\R$ for which the trust-region method requires exactly $\bar k_\epsilon$ iterations to produce $x_{\bar k_\epsilon}$ satisfying $\lvert f_{\alpha,\epsilon}'(x_{\bar k_\epsilon})\rvert\leq\epsilon$.

Set $\vartheta_\alpha\defeq\alpha^{1/(1-\alpha)}$ if $\alpha\in[0,1)$ and $\vartheta_1\defeq e^{-1}$.
For $k\in\{0,\ldots,\bar k_\epsilon\}$, define
\begin{equation}
  \label{eq:gk-alpha-sharp}
  g_k
  \defeq
  \begin{cases}
    -g_\epsilon,
     & 0\leq k<\bar k_\epsilon,
    \\[1ex]
    -\vartheta_\alpha g_\epsilon,
     & k=\bar k_\epsilon.
  \end{cases}
\end{equation}
Because $\vartheta_\alpha\leq e^{-1}<\frac12$, we have $\lvert g_k\rvert=2\epsilon>\epsilon$ for $k=0,\ldots,\bar k_\epsilon-1$, whereas $\lvert g_{\bar k_\epsilon}\rvert =2\vartheta_\alpha\epsilon<\epsilon$.

For every $k=0,\ldots,\bar k_\epsilon-1$, define
\begin{equation}
  \label{eq:Bk-alpha-sharp}
  B_k
  \defeq
  \lvert g_k\rvert^\alpha
  =
  g_\epsilon^\alpha.
\end{equation}
Thus, \Cref{asm:Bk-L0-L1} holds with $b_0=0$ and $b_1=1$.

Set $\Delta_0\geq1$ and, for every $k=0,\ldots,\bar k_\epsilon-1$, define
\begin{equation}
  \label{eq:xk-alpha-sharp}
  x_0\defeq0,
  \qquad
  s_k
  \defeq
  -B_k^{-1}g_k
  =
  g_\epsilon^{1-\alpha},
  \qquad
  x_{k+1}\defeq x_k+s_k.
\end{equation}
Consequently, $x_k=kg_\epsilon^{1-\alpha}$ for $k=0,\ldots,\bar k_\epsilon$.

Finally, prescribe the function values at the iterates by
\begin{equation}
  \label{eq:fk-alpha-sharp}
  f_k
  \defeq
  \begin{cases}
    \dfrac{g_\epsilon^{2-\alpha}}{2-\alpha}
    +(\bar k_\epsilon-1-k)g_\epsilon^{2-\alpha},
     & 0\leq k<\bar k_\epsilon,
    \\[2ex]
    \dfrac{g_\epsilon^{2-\alpha}}{2-\alpha}
    \vartheta_\alpha^{2-\alpha},
     & k=\bar k_\epsilon.
  \end{cases}
\end{equation}

The next lemma establishes basic properties of $\{f_k\}$.

\begin{lemma}%
  \label{lem:fk-alpha-sharp}
  The sequence $\{f_k\}$ defined in~\eqref{eq:fk-alpha-sharp} is non-increasing and satisfies $f_k\in[0,f_0]$ for every $k=0,\ldots,\bar k_\epsilon$, where $\bar k_\epsilon$ is defined in~\eqref{eq:def-Neps-alpha}.
  Moreover, $f_0\leq3$.
\end{lemma}

\begin{proof}
  For every $k=0,\ldots,\bar k_\epsilon-2$, $f_{k+1}-f_k=-g_\epsilon^{2-\alpha}<0$.
  For the final pair, \( f_{\bar k_\epsilon}-f_{\bar k_\epsilon-1} = \frac{g_\epsilon^{2-\alpha}}{2-\alpha} \bigl(\vartheta_\alpha^{2-\alpha}-1\bigr) <0\), because $\vartheta_\alpha<1$.
  Hence, $\{f_k\}$ is non-increasing.
  Since every term in~\eqref{eq:fk-alpha-sharp} is nonnegative, we obtain $0\leq f_k\leq f_0$ for every $k=0,\ldots,\bar k_\epsilon$.

  Finally, using~\eqref{eq:def-Neps-alpha},
  \[
    f_0
    \leq
    \frac{g_\epsilon^{2-\alpha}}{2-\alpha}
    +2g_\epsilon^{\alpha-2}g_\epsilon^{2-\alpha}
    =
    \frac{g_\epsilon^{2-\alpha}}{2-\alpha}+2
    \leq3,
  \]
  where the last inequality follows from $g_\epsilon=2\epsilon\leq1$ and $2-\alpha\geq1$.
\end{proof}

The next theorem shows that the complexity bound of order \(\epsilon^{\alpha-2}\) is sharp for \((\alpha,L_0,L_1)\)-smooth objectives with \(L_0=0\).

\begin{theorem}%
  \label{thm:sharpness-alpha}
  Let $\alpha\in[0,1]$ and $0<\epsilon\leq\frac12$.
  \Cref{alg:TR-GEN}, with model Hessians satisfying
  \Cref{asm:Bk-L0-L1} with $b_0=0$ and $b_1=1$, applied to a
  continuously differentiable, bounded-below,
  $(\alpha,0,1)$-smooth function $f\colon\R\to\R$, may require
  $\bar k_\epsilon$ iterations to produce $x_{\bar k_\epsilon}$ such that
  $\lvert f'(x_{\bar k_\epsilon})\rvert\leq\epsilon$, where
  $\bar k_\epsilon$ is defined in~\eqref{eq:def-Neps-alpha}.
\end{theorem}

\begin{proof}
  Set $\ell_{\alpha,\epsilon}\defeq g_\epsilon^{1-\alpha}$, $z_{\alpha,\epsilon}\defeq(\bar k_\epsilon-1)\ell_{\alpha,\epsilon}$, $c_{\alpha,\epsilon}\defeq g_\epsilon^{2-\alpha}/(2-\alpha)$ and set
  \(
  w_{\alpha,\epsilon}
  \defeq
  z_{\alpha,\epsilon}
  +g_\epsilon^{1-\alpha}/(1-\alpha),
  \) with the convention that $w_{\alpha,\epsilon}=+\infty$ when $\alpha=1$.
  Consider
  \begin{equation}
    \label{eq:f-alpha-sharp}
    f_{\alpha,\epsilon}(x)
    \defeq
    \begin{cases}
      c_{\alpha,\epsilon}
      +g_\epsilon(z_{\alpha,\epsilon}-x),
       & x\leq z_{\alpha,\epsilon},
      \\[1ex]
      \dfrac{1}{2-\alpha}
      \left[
      g_\epsilon^{1-\alpha}
      -(1-\alpha)(x-z_{\alpha,\epsilon})
      \right]^{\frac{2-\alpha}{1-\alpha}},
       & z_{\alpha,\epsilon}<x<w_{\alpha,\epsilon},
      \quad \alpha<1,
      \\[2ex]
      0,
       & x\geq w_{\alpha,\epsilon},
      \quad \alpha<1,
      \\[1ex]
      g_\epsilon e^{z_{\alpha,\epsilon}-x},
       & x>z_{\alpha,\epsilon},
      \quad \alpha=1.
    \end{cases}
  \end{equation}

  We now show that $f_{\alpha,\epsilon}$ is continuously differentiable and $(\alpha,0,1)$-smooth.
  Suppose first that $\alpha\in[0,1)$.
  On the intervals \(I_{\alpha, 1}=(-\infty,z_{\alpha,\epsilon}]\), \(I_{\alpha, 2}=[z_{\alpha,\epsilon},w_{\alpha,\epsilon}]\), and \(I_{\alpha, 3}=[w_{\alpha,\epsilon},+\infty)\), the function is respectively affine, polynomial-type, and constant.
  At the first junction, the adjacent pieces satisfy \(f_{\alpha,\epsilon}(z_{\alpha,\epsilon}) = c_{\alpha,\epsilon}, f_{\alpha,\epsilon}'(z_{\alpha,\epsilon}) = -g_\epsilon\).
  At the second junction, they satisfy \(f_{\alpha,\epsilon}(w_{\alpha,\epsilon}) = f_{\alpha,\epsilon}'(w_{\alpha,\epsilon}) = 0\).
  Thus, $f_{\alpha,\epsilon}\in\mathcal C^1(\R)$.
  Moreover, on $I_{\alpha, 1}$ and $I_{\alpha, 3}$ we have $\lvert f_{\alpha,\epsilon}''\rvert=0$, while on $I_{\alpha, 2}$, \(\lvert f_{\alpha,\epsilon}''(x)\rvert = \lvert f_{\alpha,\epsilon}'(x)\rvert^\alpha\).
  Hence, $f_{\alpha,\epsilon}$ is $(\alpha,0,1)$-smooth on each $I_{\alpha, i}$.
  Applying \Cref{lem:linking-alpha} first at $z_{\alpha,\epsilon}$ on $(-\infty,w_{\alpha,\epsilon}]$, and then at $w_{\alpha,\epsilon}$ on $\R$, shows that $f_{\alpha,\epsilon}$ is $(\alpha,0,1)$-smooth on $\R$.

  When $\alpha=1$, consider \(I_{1,1}=(-\infty,z_{1,\epsilon}]\) and \(I_{1,2}=[z_{1,\epsilon},+\infty)\).
  The function is affine on $I_{1,1}$ and exponential on $I_{1,2}$, and the two pieces satisfy \(f_{1,\epsilon}(z_{1,\epsilon})=g_\epsilon, f_{1,\epsilon}'(z_{1,\epsilon})=-g_\epsilon\).
  Therefore, $f_{1,\epsilon}\in\mathcal C^1(\R)$.
  Furthermore, $\lvert f_{1,\epsilon}''\rvert=0$ on $I_{1,1}$ and $\lvert f_{1,\epsilon}''\rvert =\lvert f_{1,\epsilon}'\rvert$ on $I_{1,2}$.
  An application of \Cref{lem:linking-alpha} at $z_{1,\epsilon}$ shows that $f_{1,\epsilon}$ is $(1,0,1)$-smooth on $\R$.

  We now verify that the sequences defined in~\eqref{eq:gk-alpha-sharp}--\eqref{eq:fk-alpha-sharp} constitute a valid run of \Cref{alg:TR-GEN}.
  From~\eqref{eq:Bk-alpha-sharp} and~\eqref{eq:xk-alpha-sharp}, $B_k=g_\epsilon^\alpha$, $s_k=\ell_{\alpha,\epsilon}$, and $x_k=k\ell_{\alpha,\epsilon}$.
  In particular, $x_{\bar k_\epsilon-1}=z_{\alpha,\epsilon}$, and direct substitution in~\eqref{eq:f-alpha-sharp} gives $f_{\alpha,\epsilon}(x_k)=f_k$ and $f_{\alpha,\epsilon}'(x_k)=g_k$.

  The quadratic model is $m_k(s)=f_k-g_\epsilon s+\frac12g_\epsilon^\alpha s^2$.
  Its unique minimizer is $s_k=g_\epsilon^{1-\alpha}$, and its predicted reduction is $\frac12g_\epsilon^{2-\alpha}$.
  Because $g_\epsilon\leq1$, $\lvert s_k\rvert\leq1$.
  Starting from $\Delta_0\geq1$, an induction shows that $s_k$ lies in the trust region at every iteration.
  Indeed, for $k=0,\ldots,\bar k_\epsilon-2$, the objective is affine on $[x_k,x_{k+1}]$, and hence $\rho_k=2$.

  At the final iteration, direct evaluation of~\eqref{eq:f-alpha-sharp} gives
  \[
    \rho_{\bar k_\epsilon-1}
    =
    \frac{2}{2-\alpha}
    \left(1-\vartheta_\alpha^{2-\alpha}\right)
    \geq1.
  \]
  Indeed, $\vartheta_\alpha^{2-\alpha}\leq\alpha/2$ for $\alpha\in[0,1)$, while $\vartheta_1=e^{-1}<1/2$.
  Thus, every iteration is very successful, $\Delta_{k+1}\geq\Delta_k$, and the induction is complete.
  Since $s_k$ is the exact trust-region minimizer, the required model-decrease condition also holds.

  Finally,~\eqref{eq:gk-alpha-sharp} gives $\lvert f_{\alpha,\epsilon}'(x_k)\rvert =g_\epsilon=2\epsilon>\epsilon$ for $k<\bar k_\epsilon$, whereas $\lvert f_{\alpha,\epsilon}'(x_{\bar k_\epsilon})\rvert =\vartheta_\alpha g_\epsilon<\epsilon$.
  Thus, $\bar k_\epsilon$ is the first successful termination index.
  Moreover, by~\eqref{eq:def-Neps-alpha}, $\bar k_\epsilon>2g_\epsilon^{\alpha-2} =2^{\alpha-1}\epsilon^{\alpha-2}$, which proves the result.
\end{proof}

\section{Discussion}%
\label{sec:discussion}

Several perspectives remain open.
Our analysis is deterministic and concerns first-order stationarity, extending it to stochastic objectives would require additional assumptions.
Moreover, sharpness is established only for the nonconvex case, leaving open whether the convex and strongly convex bounds are optimal.
Finally, extending the analysis to broader generalized smoothness conditions, in particular \(\alpha>1\), remains an interesting direction for future work.

\appendix
\section{Proofs}%
\label{sec:proofs}

\printProofs

\begin{lemma}%
  \label{lem:bounded-secant-slopes-L0-L1}
  Let \Cref{asm:problem-obj,asm:smooth} hold.
  Assume \(k \in \mathcal{S}\), \(s_k \neq 0\) and define \(y_k\defeq\nabla f(x_k+s_k)-\nabla f(x_k)\).
  Then, for every fixed \(r>0\),
  \[
    \frac{\|y_k\|}{\|s_k\|}
    \le
    b_m \defeq \inf_{r > 0} \max\left\{
    L_0+L_1\overline\nu_r^\alpha,
    \frac{2 g_{\max}}{r}
    \right\},
  \]
  where
  \[
    \overline\nu_r
    \defeq
    \begin{cases}
      \left[
        (g_{\max}+L_0r)^{1-\alpha}
        +(1-\alpha)L_1r
        \right]^{\frac1{1-\alpha}},
       & \alpha\in[0,1),   \\[1ex]
      \left(g_{\max}+\dfrac{L_0}{L_1}\right)e^{L_1r}
      -\dfrac{L_0}{L_1},
       & \alpha=1,\ L_1>0, \\[1ex]
      g_{\max}+L_0r,
       & \alpha=1,\ L_1=0,
    \end{cases}
  \]
  where \(g_{\max} \defeq \sqrt{8L_0\bigl(f(x_0)-f_{\textup{low}}\bigr)} + (2^{\alpha+3}L_1\bigl(f(x_0)-f_{\textup{low}}\bigr))^{\frac{1}{2-\alpha}}\).
\end{lemma}

\begin{proof}
  Let \(r > 0\) be fixed.
  Fix \(k \in \mathcal S\), and define \(\phi(t)\defeq\|\nabla f(x_k+ts_k)\|\) for \(t\in[0,1]\).
  By \Cref{lem:decrease-L0-L1}~\eqref{eq:gradient-bound}, we have \(\max\{\phi(0),\phi(1)\}\le g_{\max}\).
  By \((\alpha,L_0,L_1)\)-smoothness,
  \[
    \phi(t)
    \le
    \phi(0)
    +
    \int_0^t
    \bigl(L_0+L_1\phi(\tau)^\alpha\bigr)
    \|s_k\|\,d\tau.
  \]

  Assume first that \(\|s_k\|\le r\).
  For \(\alpha\in[0,1)\), we have
  \[
    \phi(t)
    \le
    g_{\max} + L_0r
    +
    \int_0^t L_1\|s_k\|\phi(\tau)^\alpha\,d\tau.
  \]
  Applying \Cref{lem:bihari-lasalle} with \(u=\phi\), \(c=g_{\max}+L_0r\), \(f(\tau)=L_1\|s_k\|\), and \(w(z)=z^\alpha\),
  \[
    \phi(t)
    \le
    \left[
      (g_{\max}+L_0r)^{1-\alpha}
      +(1-\alpha)L_1t\|s_k\|
      \right]^{\frac1{1-\alpha}}
    \le
    \overline\nu_r.
  \]
  For \(\alpha=1\) and \(L_1>0\), applying \Cref{lem:bihari-lasalle} with \(u=\phi\), \(c=\phi(0)\), \(f(\tau)=\|s_k\|\), and \(w(z)=L_0+L_1z\), gives
  \[
    \phi(t)
    \le
    \left(\phi(0)+\frac{L_0}{L_1}\right)
    e^{L_1t\|s_k\|}
    -
    \frac{L_0}{L_1}
    \le
    \overline\nu_r.
  \]
  If \(\alpha=1\) and \(L_1=0\), then directly \(\phi(t)\le\phi(0)+L_0t\|s_k\| \le g_{\max} + L_0r=\overline\nu_r\).

  Therefore,
  \[
    \|y_k\|
    \le
    \|s_k\|
    \sup_{t\in[0,1]}
    \bigl(L_0+L_1\phi(t)^\alpha\bigr)
    \le
    \bigl(L_0+L_1\overline\nu_r^\alpha\bigr)\|s_k\|.
  \]

  If \(\|s_k\|>r\), the endpoint bounds give \(\|y_k\|\le \|\nabla f(x_k+s_k)\|+\|\nabla f(x_k)\| \le 2 g_{\max}\).
  Combining the two cases proves the result for every fixed \(r>0\).
  Taking the infimum over \(r>0\) gives the desired result.
\end{proof}

\begin{lemma}[Bihari-LaSalle inequality \citep{bihari1956generalization}]%
  \label{lem:bihari-lasalle}
  Let \(T>0\), let \(u,f:[0,T]\to[0,\infty)\) be continuous, and let \(w:[0,\infty)\to[0,\infty)\) be continuous and nondecreasing, with \(w(r)>0\) for every \(r>0\).
  Assume that, for all \(t\in[0,T]\),
  \[
    u(t)\le c+\int_0^t f(s)w(u(s))\,ds,
  \]
  where \(c \ge 0\).
  Then
  \[
    u(t)\le G^{-1}\!\left(G(c)+\int_0^t f(s)\,ds \right), \qquad G(r)\defeq\int_{r_0}^r\frac{d\tau}{w(\tau)},
  \]
  for all \(t\in[0,T]\) for which the right-hand side is well-defined, where \(r_0>0\) is arbitrary.
\end{lemma}

\begin{lemma}%
\label{lem:linking-alpha}
Let \(I\subseteq\R\) be an interval, \(a\in I\), \(I_-:=\{x\in I:x\le a\}\) and \(I_+:=\{x\in I:x\ge a\}\).
Let \(\alpha\in[0,1]\) and \(L_0,L_1\ge0\).
Suppose that \(f\in \mathcal C^1(I)\) is \((\alpha,L_0,L_1)\)-smooth on \(I_-\) and \(I_+\). Then \(f\) is \((\alpha,L_0,L_1)\)-smooth on \(I\).
\end{lemma}

\begin{proof}
Consider \(x\in I_-\) and \(y\in I_+\).
By the triangle inequality and generalized smoothness on \(I_-\) and \(I_+\),
\[
\begin{aligned}
|f'(y)-f'(x)|
&\le |f'(y)-f'(a)|+|f'(a)-f'(x)|\\
&\le
(L_0+L_1 \sup_{u\in[a,y]}|f'(u)|^\alpha)|y-a|
+
(L_0+L_1 \sup_{u\in[x,a]}|f'(u)|^\alpha)|a-x|\\
& \le
(L_0+ L_1 \sup_{u\in[x,y]}|f'(u)|^\alpha)|y-x|.
\end{aligned}
\]
Hence, \(f\) is \((\alpha,L_0,L_1)\)-smooth on \(I\).
\end{proof}

\bibliographystyle{abbrvnat}
\bibliography{abbrv,l0l1}

\end{document}